\documentclass[12pt]{amsart} 
\usepackage{amsfonts,amsmath,amsthm,amsfonts,amscd,amssymb,amsmath,latexsym,multicol,euscript}
\usepackage{epsfig,graphicx,flafter}
\usepackage{enumerate}
\usepackage[active]{srcltx}
\usepackage{hyperref,url}
\usepackage{color}
\usepackage{asymptote}
\usepackage[small,nohug,UglyObsolete]{diagrams}
\diagramstyle[labelstyle=\scriptstyle]

\diagramstyle[noPostScript]

\makeatletter

\def\jobis#1{FF\fi
  \def\preedicate{#1}%
  \edef\preedicate{\expandafter\strip@prefix\meaning\preedicate}%
  \edef\job{\jobname}%
  \ifx\job\preedicate
}

\makeatother

\if\jobis{proposal}%
 \def\try{subsection}%
\else
  \def\try{section}%
\fi

\theoremstyle{plain}
\newtheorem{theorem}{Theorem}[\try]
\newtheorem{corollary}[theorem]{Corollary}
\newtheorem{lemma}[theorem]{Lemma}

\newtheorem{example}[theorem]{Example}
\newtheorem{proposition}[theorem]{Proposition}
\newtheorem{definition-lemma}[theorem]{Definition-Lemma}
\newtheorem{definition-proposition}[theorem]{Definition-Proposition}
\newtheorem{definition-theorem}[theorem]{Definition-Theorem}

\newtheorem{definition}[theorem]{Definition}

\newtheorem{conjecture}[theorem]{Conjecture}

\newtheorem{theorema}{Theorem}

\def\lfomitlist#1.#2.#3.#4.{{#1}_0,{#1}_1 #2 \dots #2\hat{{#1}_{#4}} #2\dots #2 {#1}_{#3}}

\def\alist#1.#2.#3.{{#1}_1 #2 {#1}_2 #2\dots #2 {#1}_{#3}}
\def\zlist#1.#2.#3.{#1_0 #2 #1_1 #2\dots #2 #1_{#3}}
\def\ltomitlist#1.#2.#3.{{#1}_0,{#1}_1 #2 \dots #2\hat {{#1}_i} #2\dots #2 {#1}_{#3}}
\def\lomitlist#1.#2.#3.{{#1}_0 #2 {#1}_1 #2 \dots #2 \hat {{#1}_i} #2 \dots #2 {#1}_{#3}}
\def\lmap#1.#2.#3.{#1 \overset{#2}{\longrightarrow} #3}
\def\mes#1.#2.#3.{#1 \longrightarrow #2 \longrightarrow #3}
\def\ses#1.#2.#3.{0\longrightarrow #1 \longrightarrow #2 \longrightarrow #3 \longrightarrow 0}
\def\les#1.#2.#3.{0\longrightarrow #1 \longrightarrow #2 \longrightarrow #3}
\def\res#1.#2.#3.{#1 \longrightarrow #2 \longrightarrow #3\longrightarrow 0}
\def\Hi#1.#2.#3.{\text {Hilb}^{#1}_{#2}(#3)}
\def\ten#1.#2.#3.{#1\underset {#2}{\otimes} #3}

\def\mderiv#1.#2.#3.{\frac {d^{#3} #1}{d #2^{#3}}}
\def\mfderiv#1.#2.#3.{\frac {\partial^{#3} #1}{\partial #2}}
\def\ggr#1.#2.#3.{\mathbb{G}_{#1}(#2,#3)}

\def\llist#1.#2.{{#1}_1,{#1}_2,\dots,{#1}_{#2}}
\def\ulist#1.#2.{{#1}^1,{#1}^2,\dots,{#1}^{#2}}
\def\lomitlist#1.#2.{{#1}_1,{#1}_2,\dots,\hat {{#1}_i}, \dots, {#1}_{#2}}
\def\lomitlistz#1.#2.{{#1}_0,{#1}_1,\dots,\hat {{#1}_i}, \dots, {#1}_{#2}}
\def\loc#1.#2.{\Cal O_{#1,#2}}
\def\fderiv#1.#2.{\frac {\partial #1}{\partial #2}}
\def\deriv#1.#2.{\frac {d #1}{d #2}}
\def\map#1.#2.{#1 \longrightarrow #2}
\def\rmap#1.#2.{#1 \dasharrow #2}
\def\emb#1.#2.{#1 \hookrightarrow #2}
\def\non#1.#2.{\text {Spec }#1[\epsilon]/(\epsilon)^{#2}}
\def\Hi#1.#2.{\text {Hilb}^{#1}(#2)}
\def\sym#1.#2.{\operatorname {Sym}^{#1}(#2)}
\def\Hb#1.#2.{\text {Hilb}_{#1}(#2)}
\def\Hm#1.#2.{\Hom_{#1}(#2)}
\def\prd#1.#2.{{#1}_1\cdot {#1}_2\cdots {#1}_{#2}}
\def\Bl #1.#2.{\operatorname {Bl}_{#1}#2}
\def\pl #1.#2.{#1^{\otimes #2}}
\def\mgn#1.#2.{\overline {M}_{#1,#2}}
\def\ialist#1.#2.{{#1}_1 #2 {#1}_2 #2 {#1}_3 #2\dots}
\def\pair#1.#2.{\langle #1, #2\rangle}
\def\gproj#1.#2.{\mathbb{P}_{#1}(#2)}
\def\gpr #1.#2.{\mathbb{P}^{#1}_{#2}}
\def\gaf #1.#2.{\mathbb{A}^{#1}_{#2}}
\def\vandermonde#1.#2.{\left|
\begin{matrix}
1 & 1 & 1 & \dots & 1\\
{#1}_1 & {#1}_2 & {#1}_3 & \dots & {#1}_{#2}\\
{#1}_1^2 & {#1}_2^2 & {#1}_3^2 & \dots & {#1}_{#2}^2\\
\vdots & \vdots & \vdots & \ddots & \vdots\\
{#1}_1^{#2-1} & {#1}_2^{#2-1} & {#1}_2^{#2-1} & \dots & {#1}_{#2}^{#2-1}\\
\end{matrix}
\right|
}
\def\vandermondet#1.#2.{\left|
\begin{matrix}
1 & {#1}_1   & {#1}_1^2 & \dots & {#1}_1^{#2-1}\\
1 & {#1}_2   & {#1}_2^2 & \dots & {#1}_2^{#2-1}\\
1 & {#1}_3   & {#1}_3^2 & \dots & {#1}_3^{#2-1}\\
\vdots & \vdots & \vdots & \ddots & \vdots\\
1 & {#1}_{#2}& {#1}_{#2}^2 & \dots & {#1}_{#2}^{#2-1}\\
\end{matrix}
\right|
}
\def\gr#1.#2.{\mathbb{G}(#1,#2)}
\def\bdd#1.#2.{{#1}_{\rfdown #2.}}

\def\ideal#1.{I_{#1}}
\def\ring#1.{\mathcal {O}_{#1}}
\def\fring#1.{\hat{\mathcal {O}}_{#1}}
\def\aring#1.{{\mathcal {O}}_{#1}^{\text{an}}}
\def\proj#1.{\mathbb {P}(#1)}
\def\pr #1.{\mathbb {P}^{#1}}
\def\dpr #1.{\hat{\mathbb {P}}^{#1}}
\def\af #1.{\mathbb{A}^{#1}}
\def\Hz #1.{\mathbb{F}_{#1}}
\def\Hbz #1.{\overline{\mathbb {F}}_{#1}}
\def\fb#1.{\underset {#1} {\times}}
\def\rest#1.{\underset {\ \ring #1.} \to \otimes}
\def\au#1.{\operatorname {Aut}\,(#1)}
\def\deg#1.{\operatorname {deg } (#1)}
\def\pic#1.{\operatorname {Pic}\,(#1)}
\def\pico#1.{\operatorname{Pic}^0(#1)}
\def\picg#1.{\operatorname {Pic}^G(#1)}
\def\ner#1.{NS (#1)}
\def\rdown#1.{\llcorner#1\lrcorner}
\def\rfdown#1.{\lfloor{#1}\rfloor}
\def\rup#1.{\ulcorner{#1}\urcorner}
\def\rfup#1.{\lceil{#1}\rceil}
\def\bp#1.{#1^{{}\leq 1}}
\def\rcup#1.{\lceil{#1}\rceil}
\def\cone#1.{\operatorname {NE}(#1)}
\def\mone#1.{\operatorname {NM}(#1)}
\def\none#1.{\operatorname {NF}(#1)}
\def\big#1.{\operatorname {B}(#1)}
\def\ccone#1.{\overline{\operatorname {NE}}(#1)}
\def\cmone#1.{\overline{\operatorname {NM}}(#1)}
\def\cnone#1.{\overline{\operatorname {NF}}(#1)}
\def\cbig#1.{\overline{\operatorname {B}(#1)}}
\def\coef#1.{\frac{(#1-1)}{#1}}
\def\vit#1.{D_{\langle #1 \rangle}}
\def\mm#1.{\overline {M}_{0,#1}}
\def\Hone#1.{H^1(#1,{\ring #1.})}
\def\ac#1.{\overline {\mathbb F}_{#1}}

\def\adj#1.{\frac {#1-1}{#1}}
\def\spn#1.{\overline{#1}}
\def\pek#1.#2.{\Cal P^{#1}(#2)}
\def\plk#1.#2.{\Cal P^{\leq #1}(#2)}
\def\ev#1.{\operatorname{ev_{#1}}}
\def\ilist#1.{{#1}_1,{#1}_2,\ldots}
\def\bminv#1.{(\nu_1,s_1;\nu_2,s_2;\dots ;\nu_{#1},s_{#1};\nu_{r+1})}
\def\zinv#1.{(\nu_1,s_1;\nu_2,s_2;\dots ;\nu_{#1},s_{#1};0)}
\def\iinv#1.{(\nu_1,s_1;\nu_2,s_2;\dots ;\nu_{#1},s_{#1};\infty)}
\def\scr#1.{\mathbf{\EuScript{#1}}}
\def\mg#1.{\overline {M}_{#1}}
\def\inter#1.{\underset #1{\cdot}}
\def\cate#1.{\text{(\underline{#1})}}
\def\dls#1.{\overrightarrow{#1}}

\def\Diff{\operatorname{Diff}}

\def\Hom{\operatorname{Hom}}

\def\sp{\operatorname{Spec}}

\def\Proj{\operatorname{Proj}}

\def\dim{\operatorname{dim}}

\def\deg{\operatorname{deg}}

\def\mult{\operatorname{mult}}

\def\rest{\operatorname{res}}

\def\vol{\operatorname{vol}}

\def\C`har{\operatorname{char}}

\def\lct{\operatorname{lct}}
\def\Lct{\operatorname{LCT}}

\def\C{\mathbb C}

\def\e{\Cal E}

\def\e1{E_1}
\def\e2{E_2}

\def\mapdown#1{\big\downarrow\rlap{$\vcenter
{\hbox{$\scriptstyle#1$}}$}}

\def\mapse#1{
{\vcenter{\hbox{$\mathop{\smash{\raise1pt\hbox{$\diagdown$}\!\lower7pt
\hbox{$\searrow$}}\vphantom{p}}\limits_{#1}\vphantom{\mapdown{}}$}}}}

\def\VR#1.{height#1pt&\omit&&\omit&&\omit&&\omit&&\omit&\cr}

\def\VRT#1.{height#1pt&\omit&&\omit&\cr}

\begin{document}
\title{ACC for log canonical thresholds}
\author{Christopher D. Hacon}
\date{\today}
\address{Department of Mathematics \\
University of Utah\\
155 South 1400 East\\
JWB 233\\
Salt Lake City, UT 84112, USA}
\email{hacon@math.utah.edu}
\author{James M\textsuperscript{c}Kernan}
\address{Department of Mathematics\\
MIT\\
77 Massachusetts Avenue\\
Cambridge, MA 02139, USA}
\email{mckernan@math.mit.edu}
\author{Chenyang Xu}
\address{Beijing International Center of Mathematics Research\\ 5 Yiheyuan Road, Haidian District\\ 
Beijing 100871, China}
\email{cyxu@math.pku.edu.cn}
\address{Department of Mathematics \\
University of Utah\\
155 South 1400 East\\
JWB 233\\
Salt Lake City, UT 84112, USA}
\email{cyxu@math.utah.edu}

\dedicatory{To Vyacheslav Shokurov on the occasion of his sixtieth birthday}

\thanks{The first author was partially supported by NSF research grant no: 0757897, the
  second author was partially supported by NSF research grant no: 0701101, and the third
  author was partially supported by NSF research grant no: 1159175 and by a special
  research fund in China.  We would like to thank Valery Alexeev, J\'anos Koll\'ar,
  and Vyacheslav Shokurov for some helpful comments.}

\begin{abstract} We show that log canonical thresholds satisfy the ACC.
\end{abstract}

\maketitle

\tableofcontents

\section{Introduction}

We work over an algebraically closed field of characteristic zero.  ACC stands for the
ascending chain condition whilst DCC stands for the descending chain condition.
  
Suppose that $(X,\Delta)$ is a log canonical pair and $M\geq 0$ is $\mathbb{R}$-Cartier.
The \textbf{log canonical threshold} of $M$ with respect to $(X,\Delta)$ is
$$
\lct(X,\Delta;M)=\sup \{\, t\in \mathbb{R}  \,|\, \text{$(X,\Delta+tM)$ is log canonical} \,\}.
$$ 
Let $\mathfrak{T}=\mathfrak{T}_n(I)$ denote the set of log canonical pairs $(X,\Delta)$,
where $X$ is a variety of dimension $n$ and the coefficients of $\Delta$ belong to a set
$I\subset [0,1]$.  Set
$$
\Lct_n(I,J)=\{\, \lct(X,\Delta;M) \,|\, (X,\Delta)\in \mathfrak{T}_n(I) \,\},
$$
where the coefficients of $M$ belong to a subset $J$ of the positive real numbers.  
\begin{theorem}[ACC for the log canonical threshold]\label{t_lct} Fix a positive integer 
$n$, $I\subset [0,1]$ and a subset $J$ of the positive real numbers.  

If $I$ and $J$ satisfy the DCC then $\Lct_n(I,J)$ satisfies the ACC.
\end{theorem} 

\eqref{t_lct} was conjectured by Shokurov \cite{Shokurov93}, see also \cite{Kollar92b} and
\cite{Kollar95}.  When the dimension is three, \cite{Kollar92b} proves that $1$ is not an
accumulation point from below and \eqref{t_lct} follows from the results of
\cite{Alexeev94}.  More recently \eqref{t_lct} was proved for complete intersections
\cite{EFM09} and even when $X$ belongs to a bounded family, \cite{EFM11}.

The log canonical threshold is an interesting invariant of the pair $(X,\Delta)$ and the
divisor $M$ which is a measure of the complexity of the singularities of the triple
$(X,\Delta;M)$.  It has made many appearances in many different forms, especially in the
case of hypersurfaces, see \cite{Kollar95}, \cite{Kollar07} and \cite{Totaro11}.  The ACC
for the log canonical threshold plays a role in inductive approaches to higher dimensional
geometry.  For example, after \cite{Birkar05}, we have the following application of
\eqref{t_lct}:
\begin{corollary}\label{c_ter} Assume termination of flips for $\mathbb{Q}$-factorial 
kawamata log terminal pairs in dimension $n-1$.

Let $(X,\Delta)$ be a kawamata log terminal pair where $X$ is a $\mathbb{Q}$-factorial
projective variety of dimension $n$.  If $K_X+\Delta$ is numerically equivalent to a
divisor $D\geq 0$ then any sequence of $(K_X+\Delta)$-flips terminates.
\end{corollary}

\eqref{t_lct} is a consequence of the following theorem, which was conjectured by Alexeev \cite{Alexeev94}
and Koll\'ar \cite{Kollar92b}:
\begin{theorem}\label{t_volume} Fix a positive integer $n$ and a set $I\subset [0,1]$ 
which satisfies the DCC.  Let $\mathfrak{D}$ be the set of log canonical pairs
$(X,\Delta)$ such that the dimension of $X$ is $n$ and the coefficients of $\Delta$ belong
to $I$.

Then there is a constant $\delta>0$ and a positive integer $m$ with the following
properties:
\begin{enumerate}
\item the set
$$
\{\, \vol(X,K_X+\Delta) \,|\, (X,\Delta)\in \mathfrak{D} \,\},
$$
also satisfies the DCC.
\end{enumerate}
Further, if $(X,\Delta)\in \mathfrak{D}$ and $K_X+\Delta$ is big, then 
\begin{enumerate}
\setcounter{enumi}{1}
\item $\vol(X,K_X+\Delta)\geq \delta$, and 
\item $\phi_{m(K_X+\Delta)}$ is birational.
\end{enumerate}
\end{theorem}

Note that, by convention, $\phi_{m(K_X+\Delta)}=\phi_{\rfdown m(K_X+\Delta).}$.
\eqref{t_volume} was proved for surfaces in \cite{Alexeev94}.  \eqref{t_volume} is a
generalisation of \cite[1.3]{HMX10}, which deals with the case that $(X,\Delta)$ is the
quotient of a smooth projective variety $Y$ of general type by its automorphism group.

One of the original motivations for \eqref{t_volume} is to prove the boundedness of the
moduli functor for canonically polarised varieties, see \cite{Kollar10a}.  We plan to
pursue this application of \eqref{t_volume} in a forthcoming paper.

To state more results it is convenient to give a simple reformulation of \eqref{t_lct}:
\begin{theorem}\label{t_simple} Fix a positive integer $n$ and a set $I\subset [0,1]$, which
satisfies the DCC.

Then there is a finite subset $I_0\subset I$ with the following properties:

If $(X,\Delta)$ is a log pair such that
\begin{enumerate}
\item $X$ is a variety of dimension $n$,   
\item $(X,\Delta)$ is log canonical, 
\item the coefficients of $\Delta$ belong to $I$, and 
\item there is a non kawamata log terminal centre $Z\subset X$ which is contained 
in every component of $\Delta$, 
\end{enumerate} 
then the coefficients of $\Delta$ belong to $I_0$.  
\end{theorem}

\eqref{t_simple} follows, cf. \cite{Shokurov93}, \cite[\S 18]{Kollaretal}, almost
immediately from the existence of divisorial log terminal modifications and from:

\begin{theorem}\label{t_num} Fix a positive integer $n$ and a set $I\subset [0,1]$, 
which satisfies the DCC.

Then there is a finite subset $I_0\subset I$ with the following properties:

If $(X,\Delta)$ is a log pair such that
\begin{enumerate}
\item $X$ is a projective variety of dimension $n$,  
\item $(X,\Delta)$ is log canonical, 
\item the coefficients of $\Delta$ belong to $I$, and 
\item $K_X+\Delta$ is numerically trivial, 
\end{enumerate} 
then the coefficients of $\Delta$ belong to $I_0$.  
\end{theorem}

We use finiteness of log canonical models to prove a boundedness result for log pairs:
\begin{theorem}\label{t_ample-bounded} Fix a positive integer $n$ and two real numbers 
$\delta$ and $\epsilon>0$.

Let $\mathfrak{D}$ be a set of log pairs $(X,\Delta)$ such that
\begin{itemize} 
\item $X$ is a projective variety of dimension $n$, 
\item $K_X+\Delta$ is ample, 
\item the coefficients of $\Delta$ are at least $\delta$, and
\item the log discrepancy of $(X,\Delta)$ is greater than $\epsilon$.
\end{itemize} 

If $\mathfrak{D}$ is log birationally bounded then $\mathfrak{D}$ is a bounded family.
\end{theorem}

Log birationally bounded is defined in \eqref{d_birationally-bounded}.  We use\
\eqref{t_num} and \eqref{t_ample-bounded} to prove some boundedness results about Fano
varieties.

\begin{corollary}\label{c_strong-bounded} Fix a positive integer $n$, a real number
$\epsilon>0$ and a set $I\subset [0,1]$ which satisfies the DCC.

Let $\mathfrak{D}$ be the set of all log pairs $(X,\Delta)$, where 
\begin{itemize} 
\item $X$ is a projective variety of dimension $n$, 
\item the coefficients of $\Delta$ belong to $I$, 
\item the log discrepancy of $(X,\Delta)$ is greater than $\epsilon$, 
\item $K_X+\Delta$ is numerically trivial, and 
\item $-K_X$ is ample.
\end{itemize} 

Then $\mathfrak{D}$ forms a bounded family.  
\end{corollary}

As a consequence we are able to prove a result on the boundedness of Fano varieties which
was conjectured by Batyrev (cf. \cite{Borisov99}):
\begin{corollary}\label{c_batyrev} Fix two positive integers $n$ and $r$. 

Let $\mathfrak{D}$ be the set of all kawamata log terminal pairs $(X,\Delta)$, where $X$
is a projective variety of dimension $n$ and $-r(K_X+\Delta)$ is an ample Cartier divisor.

Then $\mathfrak{D}$ forms a bounded family.  
\end{corollary}

\begin{definition}\label{d_fano-index} Let $(X,\Delta)$ be a log canonical pair, where $X$ 
is projective of dimension $n$ and $-(K_X+\Delta)$ is ample.  The \textbf{Fano index} of
$(X,\Delta)$ is the largest real number $r$ such that we can write
$$
-(K_X+\Delta) \sim_{\mathbb{R}} rH,
$$
where $H$ is a Cartier divisor.  

Fix a set $I\subset [0,1]$ and a positive integer $n$.  Let $\mathfrak{D}$ be the set of
log canonical pairs $(X,\Delta)$, where $X$ is projective of dimension $n$,
$-(K_X+\Delta)$ is ample and the coefficients of $\Delta$ belong to $I$.

The set 
$$
R=R_n(I)=\{\, r\in \mathbb{R} \,|\, \text{$r$ is the Fano index of $(X,\Delta)\in \mathfrak{D}$}  \,\},
$$
is called the \textbf{Fano spectrum} of $\mathfrak{D}$.
\end{definition}

\begin{corollary}\label{c_fano-index} Fix a set $I\subset [0,1]$ and a positive integer $n$.  

If $I$ satisfies the DCC then the Fano spectrum satisfies the ACC.  
\end{corollary}

\eqref{c_fano-index} was proved in dimension $2$ in \cite{Alexeev88} and for $R\cap
[n-2,\infty)$ in \cite{Alexeev91}.

Now given any set which satisfies the ACC it is natural to try to identify the
accumulation points.  \eqref{t_lct} implies that $\Lct_n(I)=\Lct_n(I,\mathbb{N})$
satisfies the ACC.  Koll\'ar, cf. \cite{Kollar95}, \cite{MP04}, \cite{Kollar08c},
conjectured that the accumulation points in dimension $n$ are log canonical thresholds in
dimension $n-1$:
\begin{theorem}\label{t_accumulation} If $1$ is the only accumulation point of $I\subset[0,1]$
and $I=I_+$ then the accumulation points of $\Lct_n(I)$ are $\Lct_{n-1}(I)-\{\,1\,\}$.  In
particular, if $I\subset \mathbb{Q}$ then the accumulation points of $\Lct_n(I)$ are
rational numbers.
\end{theorem}

See \S 3.4 for the definition of $I_+$.  \eqref{t_accumulation} was proved if $X$ is
smooth in \cite{Kollar08c}.  Note that in terms of inductive arguments it is quite useful
to identify the accumulation points, especially to know that they are rational.

Finally, recall:
\begin{conjecture}[Borisov-Alexeev-Borisov]\label{c_ab} Fix a positive integer $n$ and a
positive real number $\epsilon>0$.

Let $\mathfrak{D}$ be the set of all projective varieties $X$ of dimension $n$ such that
there is a divisor $\Delta$ where $(X,\Delta)$ has log discrepancy at least $\epsilon$ and
$-(K_X+\Delta)$ is ample.

Then $\mathfrak{D}$ forms a bounded family.
\end{conjecture}

Note that \eqref{t_lct}, \eqref{t_simple}, \eqref{t_num}, \eqref{c_ter} and
\eqref{t_accumulation} are known to follow from \eqref{c_ab}, (cf. \cite{MP04}).  Instead
we use birational boundedness of log pairs of general type cf.~ \eqref{t_volume} to prove
these results.

\section{Description of the proof}

\begin{theorema}[ACC for the log canonical threshold]\label{t_acc} Fix a 
positive integer $n$ and a set $I\subset [0,1]$, which satisfies the DCC.

Then there is a finite subset $I_0\subset I$ with the following property:

If $(X,\Delta)$ is a log pair such that
\begin{enumerate}
\item $X$ is a variety of dimension $n$,
\item $(X,\Delta)$ is log canonical, 
\item the coefficients of $\Delta$ belong to $I$, and 
\item there is a non kawamata log terminal centre $Z\subset X$ which is contained 
in every component of $\Delta$, 
\end{enumerate} 
then the coefficients of $\Delta$ belong to $I_0$.  
\end{theorema} 

\begin{theorema}[Upper bounds for the volume]\label{t_upper} Let $n\in \mathbb{N}$ and let
$I\subset [0,1)$ be a set which satisfies the DCC.  Let $\mathfrak{D}$ be the set of
kawamata log terminal pairs $(X,\Delta)$, where $X$ is projective of dimension $n$,
$K_X+\Delta$ is numerically trivial and the coefficients of $\Delta$ belong to $I$.

Then the set 
$$
\{\, \vol(X,\Delta) \,|\, (X,\Delta)\in \mathfrak{D} \,\},
$$
is bounded from above.
\end{theorema}

\begin{theorema}[Birational boundedness]\label{t_birational} Fix a positive integer $n$ and 
a set $I\subset [0,1]$, which satisfies the DCC.  Let $\mathfrak{B}$ be the set of log
canonical pairs $(X,\Delta)$, where $X$ is projective of dimension $n$, $K_X+\Delta$ is
big and the coefficients of $\Delta$ belong to $I$.

Then there is a positive integer $m$ such that $\phi_{m(K_X+\Delta)}$ is birational, for 
every $(X,\Delta)\in \mathfrak{B}$.  
\end{theorema}

\begin{theorema}[ACC for numerically trivial pairs]\label{t_numerically-trivial} 
Fix a positive integer $n$ and a set $I\subset [0,1]$, which satisfies the DCC.

Then there is a finite subset $I_0\subset I$ with the following property:

If $(X,\Delta)$ is a log pair such that
\begin{enumerate}
\item $X$ is projective of dimension $n$,  
\item the coefficients of $\Delta$ belong to $I$,
\item $(X,\Delta)$ is log canonical, and
\item $K_X+\Delta$ is numerically trivial, 
\end{enumerate} 
then the coefficients of $\Delta$ belong to $I_0$.
\end{theorema} 

The proof of Theorem~\ref{t_acc}, Theorem~\ref{t_upper}, Theorem~\ref{t_birational}, and
Theorem~\ref{t_numerically-trivial} proceeds by induction:
\begin{itemize} 
\item Theorem~\ref{t_numerically-trivial}$_{n-1}$ implies Theorem~\ref{t_acc}$_n$, cf.
\eqref{l_D-to-A}.
\item Theorem~\ref{t_numerically-trivial}$_{n-1}$ and Theorem~\ref{t_acc}$_{n-1}$ imply Theorem~\ref{t_upper}$_n$,
cf. \eqref{l_A-to-B}.
\item Theorem~\ref{t_birational}$_{n-1}$, Theorem~\ref{t_acc}$_{n-1}$, and
Theorem~\ref{t_upper}$_n$ imply Theorem~\ref{t_birational}$_n$, cf. \eqref{l_B-to-C}.
\item Theorem~\ref{t_numerically-trivial}$_{n-1}$ and Theorem~\ref{t_birational}$_n$
implies Theorem~\ref{t_numerically-trivial}$_n$, cf. \eqref{l_C-to-D}.
\end{itemize}

\subsection{Sketch of the proof}

\makeatletter
\renewcommand{\thetheorem}{\thesubsection.\arabic{theorem}}
\@addtoreset{theorem}{subsection}
\makeatother

The basic idea of the proof of \eqref{t_lct} goes back to Shokurov and 
we start by explaining this.  

Consider the following simple family of plane curve
singularities,
$$
C=(y^a+x^b=0)\subset \mathbb{C}^2,
$$
where $a$ and $b$ are two positive integers.  A priori, to calculate the log discrepancy
$c$, one should take a log resolution of the pair $(X=\mathbb{C}^2,C)$, write down the log
discrepancy of every exceptional divisor $E_i$ with respect to the pair $(X, tC)$ as a
function of $t$ and then find out the largest value $c$ of $t$ for which all of these log
discrepancies are non-negative.  However there is an easier way.  We know that when $t=c$
there is at least one divisor of log discrepancy zero (and every other divisor has
non-negative log discrepancy).  Let $\pi\colon\map Y.X.$ extract just this divisor.  To
construct $\pi$ we simply contract all other divisors on the log resolution.

Almost by definition we can write
$$
K_Y+E+c D=\pi^*(K_X+c C),
$$
where $E$ is the exceptional divisor and $D$ is the strict transform of $C$.  Restrict
both sides of this equation to $E$.  As the RHS is a pullback, we get a numerically
trivial divisor.

To compute the LHS we apply adjunction.  $E$ is a copy of $\pr 1.$.  One slightly delicate
issue is that $Y$ is singular along $E$ and the adjunction formula has to take account of
this.  In fact $\map Y.X.$ is precisely the weighted blow up of $X=\mathbb{C}^2$, with
weights $(a,b)$, in the given coordinates $x$, $y$.  There are two singular points $p$ and
$q$ of $Y$ along $C$, of index $a$ and $b$, and $D$ intersects $C$ transversally at
another point $r$.  If we apply adjunction we get
$$
(K_Y+E+c D)|_E=K_E+\left (\adj a.\right ) p+\left (\adj b.\right ) q +c r.
$$
As $(K_Y+E+c D)|_E$ is numerically trivial we have $(K_Y+E+c D)\cdot E=0$ so that
$$
-2+\adj a.+\adj b.+c=0,
$$
and so
$$
c=\frac 1a+\frac 1b.
$$

Now let us consider the general case.  As with the example above the first step is to
extract divisors of log discrepancy zero, $\pi\colon\map Y.X.$.  To construct $\pi$ we
mimic the argument above; pick a log resolution for the pair $(X,\Delta+C)$ and contract
every divisor whose log discrepancy is not zero.  The fact that we can do this in all
dimensions follows from the MMP (minimal model program), see \eqref{p_dlt} and $\pi$ is
called a divisorially log terminal modification.

The next step is the same, restrict to the general fibre of some divisor of log
discrepancy zero, see \eqref{l_global-to-local}.  There are similar formulae for the
coefficients of the restricted divisor, see \eqref{l_coefficients}.  In this way, we
reduce the problem from a local one in dimension $n$ to a global problem in dimension
$n-1$, see \S \ref{s_local-global}.  This explains how to go from
Theorem~\ref{t_numerically-trivial}$_{n-1}$ to Theorem~\ref{t_acc}$_n$, see the proof of
\eqref{l_D-to-A}.

The global problem involves log canonical pairs $(X,\Delta)$, where $X$ is projective and
$K_X+\Delta$ is numerically trivial.  One reason that the dimension one case is easy is
that there is only one possibility for $X$, $X$ must be isomorphic to $\pr 1.$.  In higher
dimensions it is not hard, running the MMP again, to reduce to the case where $X$ has
Picard number one, so that at least $X$ is a Fano variety and $\Delta$ is ample.  In this
case we perturb $\Delta$ by increasing one of its coefficients to get a kawamata log
terminal pair $(X,\Lambda)$ such that $K_X+\Lambda$ is ample.  We then exploit the fact
that some fixed multiple $m(K_X+\Lambda)$ of $K_X+\Lambda$ gives a birational map
$\phi_{m(K_X+\Lambda)}$.  By definition this means that $\phi_{\rfdown m(K_X+\Lambda).}$
is a birational map, which in particular means that $K_X+\bdd \Lambda.m.$ is big.  This
forces $\Delta\leq \bdd\Lambda.m.$ which implies that there are lots of gaps.  This
explains how to go from Theorem~\ref{t_birational}$_n$ to
Theorem~\ref{t_numerically-trivial}$_n$, see the proof of \eqref{l_C-to-D}.

It is clear then that the main thing to prove is that if $(X,\Delta)$ is a kawamata log
terminal pair, $K_X+\Delta$ is big and the coefficients of $\Delta$ belong to a DCC set
then some fixed multiple of $K_X+\Delta$, gives a birational map $\phi_{m(K_X+\Delta)}$.
Following some ideas of Tsuji, we developed a fairly general method to prove such a result
in \cite{HMX10}, see \eqref{t_birationally-bounded} and \eqref{t_inductive}.  We use the
technique of cutting non kawamata log terminal centres as developed in \cite{AS95}, see
\cite{Kollar95}.  The main issue is to find a boundary on the non kawamata log terminal
centre so that we can run an induction.

There are two key hypotheses to apply \eqref{t_inductive}.  One of them requires that the
volume of $K_X+\Delta$ restricted to appropriate non kawamata log terminal centres is
bounded from below.  The other places a requirement on the coefficients of $\Delta$ which
is stronger than the DCC.

The first condition follows by induction on the dimension and a strong version of
Kawamata's subadjunction formula, \eqref{t_coefficients}, which we now explain.  If
$(X,\Lambda)$ is a log pair and $V$ is a non kawamata log terminal centre such that
$(X,\Lambda)$ is log canonical at the generic point of $V$, then one can write
$$
(K_X+\Lambda)|_W=K_W+\Theta_b+J,
$$
where $W$ is the normalisation of $V$, $\Theta_b$ is the discriminant divisor and $J$ is
the moduli part.  Not much is known about the moduli part $J$ beyond the fact that it is
pseudo-effective.  On the other hand $\Theta_b\geq 0$ behaves very well.  If $(X,\Lambda)$
is log canonical at the generic point of a prime divisor $B$ on $W$ then the coefficient
of $B$ in $\Theta_b$ is at most one.  In fact there is a simple way to compute the
coefficient of $B$ involving the log canonical threshold.  By assumption there is a log
canonical place, that is, a valuation with centre $V$ of log discrepancy zero.  Then we
can find a divisorially log terminal modification $g\colon\map Y.X.$ such that the centre
of this log canonical place is a divisor $S$ on $Y$.  Note that there is a commutative
diagram
$$
\begin{diagram}
 S  &   \rTo    & Y   \\
\dTo^f &  & \dTo^g \\
 W  &   \rTo    & X.
\end{diagram}
$$
If we pullback $K_X+\Delta$ to $Y$ and restrict to $S$ we get a divisor $\Phi'$ on $S$.
Let
\begin{multline*}
\lambda=\sup\{\, t\in \mathbb{R} \,|\, \text{$(S,\Phi'+tf^*B)$ is log canonical over a}\\
\text{neighbourhood of the generic point of $B$}\,\},
\end{multline*}
be the log canonical threshold.  Then the coefficient of $B$ in $\Theta_b$ is $1-\lambda$.  

In practice we start with a divisor $\Delta$ whose coefficients belong to $I$ such that
$(X,\Delta)$ is kawamata log terminal.  We then find a divisor $\Delta_0$, whose
coefficients we have no control on, and $V$ is a non kawamata log terminal centre of
$(X,\Lambda=\Delta+\Delta_0)$.  It follows that the coefficients of $\Phi'$ do not behave
well and we have no control on the coefficients of $\Theta_b$.

To circumvent this we simply mimic the same construction for $(X,\Delta)$ rather than
$(X,\Lambda)$.  First we construct a divisor $\Phi$ on $S$ whose coefficients of $\Phi$
belong to $D(I)$, see \eqref{l_coefficients}.  Then we construct a divisor $\Theta$ whose
coefficients automatically belong to the set
$$
\{\, a \,|\, 1-a\in \Lct_{n-1}(D(I)) \,\}\cup \{\,1\,\}.
$$
It is clear from the construction that $\Theta_b\geq\Theta$, so that if we bound the
volume of $K_W+\Theta$ from below we bound the volume of $(K_X+\Delta+\Delta_0)|_W$ 
from below.  

On the other hand, as part of the induction we assume that Theorem~\ref{t_acc}$_{n-1}$
holds.  Hence $\Lct_{n-1}(D(I))$ satisfies the ACC and the coefficients of $\Theta$ belong
to a set which satisfies the DCC.  The final step is to observe that if we choose $V$ to
pass through a general point then it belongs to a family which covers $X$.  If we assume
that $V$ is a general member of such a family then we can pullback $K_X+\Delta$ to this
family and restrict to $V$.  It is straightforward to check that the difference between
$K_W+\Theta$ and $(K_X+\Delta)|_W$ on a log resolution of the family is pseudo-effective
(for example, if $X$ and $V$ are smooth then this follows from the fact that the first
chern class of the normal bundle is pseudo-effective), so that if $K_X+\Delta$ is big then
so is $K_W+\Theta$.  In this case we know the volume is bounded from below by induction.

We now explain the condition on the coefficients.  To apply \eqref{t_inductive} we require
that either $I$ is a finite set or
$$
I=\{\, \frac{r-1}r \,|\, r\in \mathbb{N} \,\}.
$$
The first lemma, \eqref{l_klt-case}, simply assumes this condition on $I$ and we deduce
the result in this case.

The key is then to reduce to the case when $I$ is finite.  Given any positive integer $p$
and a log pair $(X,\Delta)$, let $\bdd \Delta.p.$ denote the largest divisor less than
$\Delta$ such that $p\bdd \Delta.p.$ is integral.  Given $I$ it suffices to find a fixed
positive integer $p$ such that if we start with $(X,\Delta)$ such that $K_X+\Delta$ is big
and the coefficients belong to $I$ then $K_X+\bdd \Delta.p.$ is big, since the
coefficients of $\bdd \Delta.p.$ belong to the finite set
$$
\{\, \frac ip \,|\, 1\leq i\leq p \,\}.
$$
Let 
$$
\lambda=\inf \{\, t\in \mathbb{R} \,|\, \text{$K_X+t\Delta$ is big}\,\},
$$
be the pseudo-effective threshold.  A simple computation, \eqref{l_B-to-C}, shows that it
suffices to bound $\lambda$ away from one.  Running the MMP we reduce to the case when $X$
has Picard number one.  Since $K_X+\lambda\Delta$ is numerically trivial and kawamata log
terminal, Theorem~\ref{t_upper} implies that the volume of $\Delta$ is bounded away from
one.  Passing to a log resolution we may assume that $(X,D)$ has simple normal crossings
where $D$ is the sum of the components of $\Delta$.  As $K_X+D$ is big then so is
$K_X+\frac{r-1}rD$ for any positive integer $r$ which is sufficiently large.  It follows
that some fixed multiple of $K_X+\frac{r-1}rD$ gives a birational map, and
\eqref{t_birationally-bounded} implies that $(X,D)$ belongs to log birationally bounded
family.  In this case, it is easy to bound the pseudo-effective threshold $\lambda$ away
from one, see \eqref{l_psef}.  This explains how to go from Theorem~\ref{t_upper}$_n$ to
Theorem~\ref{t_birational}$_n$, cf. \eqref{l_B-to-C}.

We now explain the last implication.  Suppose that $(X,\Delta)$ is kawamata log terminal
and $K_X+\Delta$ is numerically trivial.  If the volume of $\Delta$ is large then we may
find a divisor $\Pi$ numerically equivalent to a small multiple of $\Delta$ with large
multiplicity at a general point, so that $(X,\Pi)$ is not kawamata log terminal.  In
particular we may find $\Phi$ arbitrarily close to $\Delta$ such that $(X,\Phi)$ is not
kawamata log terminal.  The key lemma is to show that this is impossible,
\eqref{l_epsilon}.  By assumption we may extract a divisor $S$ of log discrepancy zero
with respect to $(X,\Phi)$.  After we run the MMP we get down a log pair $(Y,S+\Gamma)$
where $\Gamma$ is the strict transform of $\Delta$ and both $K_Y+S+\Gamma$ and
$-(K_Y+S+(1-\epsilon)\Gamma)$ are ample.  Here $\epsilon>0$ is arbitrarily to zero.  If we
restrict to $S$ and apply adjunction, it is easy to see that this contradicts either ACC
for the log canonical threshold or ACC for numerically trivial pairs.  This explains how
to go from Theorem~\ref{t_numerically-trivial}$_{n-1}$ and Theorem~\ref{t_acc}$_{n-1}$ to
Theorem~\ref{t_upper}$_n$, cf. \eqref{l_A-to-B}.

It is interesting to note that if $(X,\Delta)$ is log canonical then there is no bound on
the volume of $\Delta$:
\begin{example} Let $X$ be the weighted projective surface $\mathbb{P}(p,q,r)$, where $p$,
$q$ and $r$ are three positive integers and let $\Delta$ be the sum of the three
coordinate lines.  Then $K_X+\Delta\sim_{\mathbb{Q}} 0$ and
$$
\vol(X,\Delta)=\frac{(p+q+r)^2}{pqr}.
$$
But the set 
$$
\{\, \frac{(p+q+r)^2}{pqr} \,|\, (p,q,r)\in \mathbb{N}^3 \,\},
$$
is dense in the positive real numbers, cf. \cite[22.5]{KM99}.
\end{example}

We now explain the proof of \eqref{t_accumulation} which mirrors the proof of
\eqref{t_lct}.  We are given a sequence of log pairs $(X,\Delta)=(X_i,\Delta_i)$ and we
want to identify the limit points of the log canonical thresholds.  The first step is to
show that the set of log canonical thresholds is essentially the same as the set of
pseudo-effective thresholds.  In \S 5 we showed that every log canonical threshold in
dimension $n+1$ is a numerically trivial threshold in dimension $n$.  To show the reverse
inclusion, one takes the cone $(Y,\Gamma)$ over a log canonical pair $(X,\Delta)$ where
$K_X+\Delta$ is numerically trivial, \eqref{p_local-global}.

In this way we are reduced to looking at log canonical pairs $(X,\Delta)$ such that
$K_X+\Delta$ is numerically trivial.  The basic idea is to generate a component of
coefficient one and apply adjunction.  To this end, we need to deal with the case where
some coefficients of $\Delta$ don't necessarily belong to $I$ but instead they are
increasing towards one, \eqref{p_stronger}.  

Running the MMP we reduce to the case of Picard number one, Case A, Step 1 and Case B,
Steps 3 and 5.  We may also assume that the non kawamata log terminal locus is a divisor.
In particular $-K_X$ is ample, any two components of $\Delta$ intersect and we may assume
that the number of components of $\Delta$ is constant, \eqref{l_number}.  If $(X,\Delta)$
is not kawamata log terminal then there is a component of coefficient one and we are done,
Case B, Step 2.

The argument now splits into two cases.  Case A deals with the case that the coefficients
of $\Delta$ are bounded away from one.  In this case if the volume of $\Delta$ is
arbitrarily large then we can create a component of coefficient one and we reduce to the
other case, Case B.  Otherwise \eqref{t_ample-bounded} implies that $(X,\Delta)$ belongs
to a bounded family, which contradicts the fact that the coefficients of $\Delta$ are not
constant.

So we may assume we are in Case B, namely that some of the coefficients of the components
of $\Delta$ are approaching one.  We decompose $\Delta$ as $A+B+C$ where the coefficients
of $A$ are approaching one, the coefficients of $B$ are fixed, and we are trying to
identify the limit of the coefficients of $C$.  Using the fact that the Picard number of
$X$ is one, we may increase the coefficients of $A$ to one and decrease the coefficients
of $C$, without changing the limit of the coefficients of $C$.  At this point we apply
adjunction and induction, Case B, Step 6.

\section{Preliminaries}

\subsection{Notation and Conventions} If $D=\sum d_iD_i$ is an $\mathbb{R}$-divisor on a
normal variety $X$, then the \textit{round down} of $D$ is $\rfdown D.=\sum \rfdown
d_i.D_i$, where $\rfdown d.$ denotes the largest integer which is at most $d$, the
\textit{fractional part} of $D$ is $\{D\}=D-\rfdown D.$, and the \textit{round up} of $D$
is $\rfup D. =-\rfdown -D.$.  If $m$ is a positive integer, then let
$$
\bdd D.m.=\frac{\rfdown mD.}m.
$$
Note that $\bdd D.m.$ is the largest divisor less than or equal to $D$ such that $m\bdd
D.m.$ is integral.

The sheaf $\ring X.(D)$ is defined by
$$
\ring X.(D)(U)=\{\, f\in K(X) \,|\, (f)|_U+D|_U\geq 0 \,\},
$$
so that $\ring X.(D)=\ring X.(\rfdown D.)$.  Similarly we define $|D|=|\rfdown D.|$.  If
$X$ is normal, and $D$ is an $\mathbb{R}$-divisor on $X$, the \textit{rational map
  $\phi_D$ associated to $D$} is the rational map determined by the restriction of
$\rfdown D.$ to the smooth locus of $X$.

We say that $D$ is \textit{$\mathbb{R}$-Cartier} if it is a real linear combination of
Cartier divisors.  An $\mathbb{R}$-Cartier divisor $D$ on a normal variety $X$ is
\textit{nef} if $D \cdot C \geq 0$ for any curve $C \subset X$.  We say that two
$\mathbb{R}$-divisors $D_1$ and $D_2$ are \textit{$\mathbb{R}$-linearly equivalent},
denoted $D_1 \sim_{\mathbb{R}}D_2$, if the difference is an $\mathbb{R}$-linear
combination of principal divisors.

A \textit{log pair} $(X,\Delta)$ consists of a normal variety $X$ and a $\mathbb{R}$-Weil
divisor $\Delta\geq 0$ such that $K_X+\Delta$ is $\mathbb{R}$-Cartier.  The
\textit{support} of $\Delta=\sum_{i\in I}d_iD_i$ (where $d_i\neq 0$) is the sum
$D=\sum_{i\in I}D_i$.  If $(X,\Delta)$ has simple normal crossings, a \textit{stratum} of
$(X,\Delta)$ is an irreducible component of the intersection $\cap_{j\in J}D_j$, where $J$
is a non-empty subset of $I$ (in particular, a stratum of $(X,\Delta)$ is always a proper
closed subset of $X$).  If we are given a morphism $\map X.T.$, then we say that
$(X,\Delta)$ has \textit{simple normal crossings over $T$} if $(X,\Delta)$ has simple
normal crossings and both $X$ and every stratum of $(X,D)$ is smooth over $T$.  We say
that the birational morphism $f\colon\map Y.X.$ only \textit{blows up strata} of
$(X,\Delta)$, if $f$ is the composition of birational morphisms $f_i\colon\map
X_{i+1}.X_i.$, $1\leq i\leq k$, with $X=X_0$, $Y=X_{k+1}$, and $f_i$ is the blow up of a
stratum of $(X_i,\Delta_i)$, where $\Delta_i$ is the sum of the strict transform of
$\Delta$ and the exceptional locus.

A \textit{log resolution} of the pair $(X,\Delta)$ is a projective birational morphism
$\mu\colon\map Y.X.$ such that the exceptional locus is the support of a $\mu$-ample
divisor and $(Y,G)$ has simple normal crossings, where $G$ is the support of the strict
transform of $\Delta$ and the exceptional divisors.  If we write
$$
K_Y+\Gamma+\sum b_iE_i=\mu^*(K_X + \Delta)
$$
where $\Gamma$ is the strict transform of $\Delta$, then $b_i$ is called the
\textit{coefficient of $E_i$ with respect to $(X,\Delta)$}.  The \textit{log discrepancy}
of $E_i$ is $a(E_{i},X,\Delta)=1-b_{i}$.  The pair $(X,\Delta)$ is \textit{kawamata log
  terminal} (respectively \textit{log canonical}; \textit{purely log terminal};
\textit{divisorially log terminal}) if $b_i<1$ for all $i$ and $\rfdown\Delta.=0$
(respectively $b_i\leq 1$ for all $i$ and for all log resolutions; $b_i<1$ for all $i$ and
for all log resolutions; the coefficients of $\Delta$ belong to $[0,1]$ and there exists a
log resolution such that $b_i<1$ for all $i$).  A \textit{non kawamata log terminal
  centre} is the centre of any valuation associated to a divisor $E_i$ with $b_i\geq 1$.
In this paper, we only consider valuations $\nu$ of $X$ whose centre on some birational
model $Y$ of $X$ is a divisor.

We now introduce some results some of which are well known to experts but which are
included for the convenience of the reader.

\subsection{The volume}

\begin{definition}\label{d_volume} Let $X$ be an irreducible projective variety
of dimension $n$ and let $D$ be an $\mathbb{R}$-divisor.  The \textbf{volume} of $D$ is
$$
\vol(X,D)=\limsup_{m\to \infty}\frac{n! h^0(X,\ring X.(mD))}{m^n}.
$$
We say that $D$ is \textbf{big} if $\vol(X,D)>0$.
\end{definition}

For more background, see \cite{Lazarsfeld04b}.  

\subsection{Divisorially log terminal modifications}

If $(X,\Delta)$ is not kawamata log terminal then we may find a modification which is
divisorially log terminal, so that the non kawamata log terminal locus is a divisor:
\begin{proposition}\label{p_dlt} Let $(X,\Delta)$ be a log pair where $X$ is a variety and
the coefficients of $\Delta$ belong to $[0,1]$.

Then there is a projective birational morphism $\pi\colon\map Y.X.$ such that
\begin{enumerate}
\item $Y$ is $\mathbb{Q}$-factorial,
\item $\pi$ only extracts divisors of log discrepancy at most zero, 
\item if $E=\sum E_i$ is the sum of the $\pi$-exceptional divisors and $\Gamma$ is the
strict transform of $\Delta$, then $(Y,\Gamma+E)$ is divisorially log terminal and
$$
K_Y+E+\Gamma=f^*(K_X+\Delta)+\sum_{a(E,X,B)<0} a(E,X,B)E.
$$
\item If in addition $(X,\Delta)$ is log canonical and the coefficients of $\Delta$ are
less than one, then we may choose $\pi$ so that there is a divisor with support equal to
$E$ which is nef over $X$.  In particular the inverse image of the non kawamata log
terminal locus of $(X,\Delta)$ is equal to the support of $E$.
\end{enumerate}
\end{proposition}
\begin{proof} The proof of (1--3) is due to the first author and can be found in
\cite{Fujino09}, \cite[3.1]{KK09}, and also \cite{AH11}.  

Now suppose that $(X,\Delta)$ is log canonical and the coefficients of $\Delta$ are less
than one.  In this case
$$
K_Y+E+\Gamma=f^*(K_X+\Delta).
$$
Note that $(Y,\Gamma)$ is kawamata log terminal and $\Gamma$ is big over $X$.  By
\cite{BCHM06} we may replace $Y$ by a log terminal model of $(Y,\Gamma)$ over $X$, at the
expense of temporarily losing the property that $(Y,\Gamma+E)$ is divisorially log
terminal.  Hence $-E$ is nef over $X$.  If $g\colon \map W.Y.$ is a divisorially log
terminal modification of $(Y,\Gamma+E)$, then $g$ is an isomorphism outside of the support
of $E$, as $(Y,\Gamma)$ is kawamata log terminal and $Y$ is $\mathbb{Q}$-factorial.
Therefore if we replace $Y$ by $W$ then $g^*(-E)$ is a nef divisor whose support is equal
to the sum of the exceptional divisors over $X$.

Note that if $x\in X$ is a point of $X$ then either $E$ contains the whole fibre over $x$
or $x\notin \pi(E)$.  But the non kawamata log terminal locus of $X$ is equal to the image
of the non kawamata log terminal locus of $(Y,\Gamma+E)$, that is, the image of $E$ and so
the inverse image of the non kawamata log terminal locus of $(X,\Delta)$ is equal to the
support of $E$.
\end{proof}

\subsection{DCC sets}\label{sub_dcc}

We say that a set $I$ of real numbers satisfies the \textit{descending chain condition} or
DCC, if it does not contain any infinite strictly decreasing sequence.  For example, 
$$
I=\{\, \frac{r-1}r \,|\, r\in \mathbb{N}\,\},
$$
satisfies the DCC.  Let $I\subset [0,1]$. We define
$$
I_+:=\{\,0\,\}\cup \{\, j\in[0,1] \,|\, \text{$j=\sum_{p=1}^l i_p$, for some $\llist i.l.\in I$} \,\},
$$ 
and
$$
D(I):=\{\, a\leq 1\,|\, a=\frac{m-1+f} m, m\in \mathbb{N},f\in I_+\,\}.
$$

As usual, $\overline{I}$ denotes the closure of $I$.  
 
\begin{proposition}\label{p_der} Let $I\subset [0,1]$.  
\begin{enumerate}
\item $D(D(I))=D(I)\cup \{\,1\,\}$.
\item $I$ satisfies the DCC if and only if $\overline{I}$ satisfies the DCC.
\item $I$ satisfies the DCC if and only if $D(I)$ satisfies the DCC.
\end{enumerate}
\end{proposition}
\begin{proof} Straightforward, see for example \cite[4.4]{MP04}.
\end{proof}

\subsection{Bounded pairs}

We recall some results and definitions from \cite{HMX10}, stated in a convenient form.

\begin{definition}\label{d_birationally-bounded} We say that a set $\mathfrak{X}$
of varieties is \textbf{birationally bounded} if there is a projective morphism $\map
Z.T.$, where $T$ is of finite type, such that for every $X\in \mathfrak{X}$, there is a
closed point $t\in T$ and a birational map $f\colon\rmap Z_t.X.$.

We say that a set $\mathfrak{D}$ of log pairs is \textbf{log birationally bounded}
(respectively \textbf{bounded}) if there is a log pair $(Z,B)$, where the coefficients of
$B$ are all one, and a projective morphism $\map Z.T.$, where $T$ is of finite type, such
that for every $(X,\Delta)\in \mathfrak{D}$, there is a closed point $t\in T$ and a
birational map $f\colon\rmap Z_t.X.$ (respectively isomorphism of pairs) such that the support of
$B_t$ contains the support of the strict transform of $\Delta$ and any $f$-exceptional
divisor.  
\end{definition}

\begin{theorem}\label{t_birationally-bounded} Fix a positive integer $n$ and a 
set $I\subset [0,1]\cap \mathbb{Q}$, which satisfies the DCC.  Let $\mathfrak{B}_0$ be a
set of log canonical pairs $(X,\Delta)$, where $X$ is projective of dimension $n$,
$K_X+\Delta$ is big and the coefficients of $\Delta$ belong to $I$.

Suppose that there is a constant $M$ such that for every $(X,\Delta)\in \mathfrak{B}_0$
there is a positive integer $k$ such that $\phi_{k(K_X+\Delta)}$ is birational and
$$
\vol(X,k(K_X+\Delta))\leq M.
$$

Then the set 
$$
\{\, \vol(X,K_X+\Delta) \,|\, (X,\Delta)\in \mathfrak{B}_0 \,\},
$$
satisfies the DCC. 
\end{theorem}
\begin{proof} Follows from (2.3.4), (3.1) and (1.9) of \cite{HMX10}.
\end{proof}

Recall:
\begin{definition}\label{d_potential} Let $X$ be a normal projective variety and let
$D$ be a big $\mathbb{Q}$-Cartier $\mathbb{Q}$-divisor on $X$.

If $x$ and $y$ are two general points of $X$ then, possibly switching $x$ and $y$, we may
find $0\leq \Delta\sim_{\mathbb{Q}}(1-\epsilon) D$, for some $0<\epsilon <1$, where
$(X,\Delta)$ is not kawamata log terminal at $y$, $(X,\Delta)$ is log canonical at $x$ and
$\{x\}$ is a non kawamata log terminal centre, then we say that $D$ is \textbf{potentially
  birational}.
\end{definition}

Note that this is a slight variation on the definition which appears in \cite{HMX10},
where general is replaced by very general.

\begin{theorem}\label{t_recursive} Let $(X,\Delta)$ be a kawamata log terminal pair, where
$X$ is projective of dimension $n$ and let $H$ be an ample $\mathbb{Q}$-divisor.  Suppose
there is a constant $\gamma\geq 1$ and a family of subvarieties $\map V.B.$ with the
following property.

If $x$ and $y$ are two general points of $X$ then, possibly switching $x$ and $y$, we can
find $b\in B$ and $0\leq \Delta_b \sim_{\mathbb{Q}} (1-\delta)H$, for some $\delta>0$,
such that $(X,\Delta+\Delta_b)$ is not kawamata log terminal at $y$ and there is a unique
non kawamata log terminal place of $(X,\Delta+\Delta_b)$ containing $x$ whose centre is
$V_b$.  Further there is a divisor $D$ on $W$, the normalisation of $V_b$, such that
$\phi_D$ is birational and $\gamma H|_W-D$ is pseudo-effective.

Then $mH$ is potentially birational, where $m=2p^2\gamma+1$ and $p=\dim V$.
\end{theorem}
\begin{proof} Let $x$ and $y$ be two general points of $X$.  Possibly switching $x$ and
$y$, we will prove by descending induction on $k$ that there is a $\mathbb{Q}$-divisor
$\Delta_0\geq 0$ such that:

\noindent ($\flat$)$_k$ $\Delta_0\sim_{\mathbb{Q}} \lambda H$, for some $\lambda <
2(p-k)p\gamma+1$, where $(X,\Delta+\Delta_0)$ is log canonical at $x$, not kawamata log
terminal at $y$ and there is a non kawamata log terminal centre $Z\subset V_b$ of
dimension at most $k$ containing $x$.

Suppose $k=p$.  $(X,\Delta+\Delta_b)$ is not kawamata log terminal but log canonical at
$x$ since there is a unique non kawamata log terminal place whose centre contains $x$.
Thus $\Delta_0=\Delta_b \sim_{\mathbb{Q}} \lambda H$, where $\lambda=1-\delta<1$,
satisfies ($\flat$)$_k$ and so this is the start of the induction.

Now suppose that we may find a $\mathbb{Q}$-divisor $\Delta_0$ satisfying ($\flat$)$_k$.
We may assume that $Z$ is the minimal non kawamata log terminal centre containing $x$ and
that $Z$ has dimension $k$.  Let $Y\subset W$ be the inverse image of $Z$.  As $\phi_D$ is
birational,
$$
\vol(Y,\gamma H|_Y)\geq \vol(Y,D|_Y)\geq 1,
$$
see, for example, \cite[2.3.2]{HMX10}.  Note that
$$
\vol(Z,\gamma H|_Z)=\vol(Y,\gamma H|_Y),
$$
as $H$ is nef, see for example \cite[VI.2.15]{Kollar96}.  Thus
$$
\vol(Z,2p\gamma H|_V)>\vol(Z,2k\gamma H|_V) \geq 2k^k,
$$
so that by \cite[2.3.5]{HMX10}, we may find $\Delta_1\sim_{\mathbb {Q}} \mu H$, where
$\mu<2p\gamma$ and constants $0<a_i\leq 1$ such that $(X,\Delta+a_0\Delta_0+a_1\Delta_1)$
is log canonical at $x$, not kawamata log terminal at $y$ and there is a non kawamata log
terminal centre $Z'$ containing $x$, whose dimension is less than $k$.  As
$$
a_0\Delta_0+a_1\Delta_1 \sim_{\mathbb{Q}} (a_0\lambda+a_1\mu)H,
$$
and 
$$
\lambda'=a_0\lambda+a_1\mu < 2(p-k)p\gamma+ 1 + 2p\gamma=2(p-(k-1))p\gamma+1,
$$
$a_0\Delta_0+a_1\Delta_1$ satisfies ($\flat$)$_{k-1}$.  This completes the induction and
the proof.
\end{proof}

\begin{theorem}\label{t_inductive} Fix a positive integer $n$.  Let $\mathfrak{B}_0$ 
be a set of kawamata log terminal pairs $(X,\Delta)$, where $X$ is projective of dimension
$n$ and $K_X+\Delta$ is ample.

Suppose that there are positive integers $p$, $k$ and $l$ such that for every
$(X,\Delta)\in \mathfrak{B}_0$ we have:
\begin{enumerate} 
\item There is a dominant family of subvarieties $\map V.B.$ such that if $b\in B$ then we
may find $0\leq \Delta_b \sim_{\mathbb{Q}} (1-\delta)H$, for some $\delta>0$, such that
there is a unique non kawamata log terminal place of $(X,\Delta+\Delta_b)$ whose centre is
$V_b$, where $H=k(K_X+\Delta)$.  Further there is a divisor $D$ on $W$ the normalisation
of $V_b$ such that $\phi_D$ is birational and $lH|_W-D$ is pseudo-effective.
\item Either $p\Delta$ is integral or the coefficients of $\Delta$ belong to 
$$
\{\, \frac{r-1}r \,|\, r\in \mathbb{N} \,\}.
$$
\end{enumerate} 

Then there is a positive integer $m$ such that $\phi_{mk(K_X+\Delta)}$ is birational, for
every $(X,\Delta)\in \mathfrak{B}_0$.
\end{theorem}
\begin{proof} Let $m_0=2(n-1)^2l+1$.  \eqref{t_recursive} implies that $m_0H$ is
potentially birational.  But then \cite[2.3.4.1]{HMX10} implies that $\phi_{K_X+\rfup
  m_0jH.}$ is birational for all positive integers $j$.

If $p\Delta$ is integral then
$$
K_X+\rfup m_0kp(K_X+\Delta).=\rfdown (m_0kp+1)(K_X+\Delta).,
$$
and if the coefficients of $\Delta$ belong to 
$$
\{\, \frac{r-1}r \,|\, r\in \mathbb{N} \,\},
$$
then 
$$
K_X+\rfup m_0kp(K_X+\Delta).=\rfdown (m_0kp+1)(K_X+\Delta)..
$$
Let $m=(m_0+1)p$.  \end{proof}

\makeatletter
\renewcommand{\thetheorem}{\thesection.\arabic{theorem}}
\@addtoreset{theorem}{section}
\makeatother

\section{Adjunction}
\label{s_adjunction}

We will need the following basic result about adjunction (see for example \S 16 in
\cite{Kollaretal}).
\begin{lemma}\label{l_coefficients} Let $(X,\Delta=S+B)$ be a log canonical pair, 
where $S$ has coefficient one in $\Delta$.  If $S$ is normal then there is a divisor
$\Theta=\Diff_S(B)$ on $S$ such that
$$
(K_X+\Delta)|_S=K_S+\Theta.
$$
\begin{enumerate}
\item If $(X,\Delta)$ is purely log terminal then $(S,\Theta)$ is kawamata log terminal.  
\item If $(X,\Delta)$ is divisorially log terminal then $(S,\Theta)$ is divisorially log terminal.  
\item If $B=\sum b_iB_i$ then the coefficients of $\Theta$ belong to the set $D(\{\, \llist b.m. \,\})$.
\end{enumerate}
In particular, if $(X,\Delta)$ is divisorially log terminal and the coefficients of $B$
belong to the set $I$ then the coefficients of $\Theta$ belong to the set $D(I)$.
\end{lemma}

\begin{theorem}\label{t_coefficients} Let $I$ be a subset of $[0,1]$ which contains $1$.  
Let $X$ be a projective variety of dimension $n$ and let $V$ be an irreducible closed
subvariety, with normalisation $W$.  Suppose we are given a log pair $(X,\Delta)$ and an
$\mathbb{R}$-Cartier divisor $\Delta'\geq 0$, with the following properties:
\begin{enumerate} 
\item the coefficients of $\Delta$ belong to $I$,  
\item $(X,\Delta)$ is kawamata log terminal, and
\item there is a unique non kawamata log terminal place $\nu$ for $(X,\Delta+\Delta')$,
with centre $V$.
\end{enumerate} 

Then there is a divisor $\Theta$ on $W$ whose coefficients belong to
$$
\{\, a \,|\, 1-a\in \Lct_{n-1}(D(I)) \,\}\cup \{\,1\,\},
$$
such that the difference
$$
(K_X+\Delta+\Delta')|_W-(K_W+\Theta),
$$ 
is pseudo-effective.  

Now suppose that $V$ is the general member of a covering family of subvarieties of $X$.
Let $\psi\colon\map U.W.$ be a log resolution of $W$ and let $\Psi$ be the sum of the
strict transform of $\Theta$ and the exceptional divisors.  Then
$$
K_U+\Psi\geq (K_X+\Delta)|_U.
$$
\end{theorem}
\begin{proof} Since there is a unique non kawamata log terminal place with centre $V$, it
follows that $(X,\Delta+\Delta')$ is log canonical but not kawamata log terminal at the
generic point of $V$, see (2.31) of \cite{KM98}.  Let $g\colon\map Y.X.$ be a divisorially
log terminal modification of $(X,\Delta+\Delta')$, \eqref{p_dlt}, so that the centre of
$\nu$ is a divisor $S$ on $Y$ and this is the only exceptional divisor with centre $V$.
As $S$ is normal, there is a commutative diagram
$$
\begin{diagram}
 S  &   \rTo    & Y   \\
\dTo^f &  & \dTo^g \\
 W  &   \rTo    & X.
\end{diagram}
$$
We may write
$$
K_Y+S+\Gamma=g^*(K_X+\Delta)+E \qquad \text{and} \qquad K_Y+S+\Gamma+\Gamma'=g^*(K_X+\Delta+\Delta'),
$$
where $\Gamma$ is the sum of the strict transform of $\Delta$ and the exceptional
divisors, apart from $S$.  In particular the coefficients of $\Gamma$ belong to $I$.  As
$(X,\Delta)$ is kawamata log terminal, $E\geq 0$.  As $g$ is a divisorially log terminal
modification of $(X,\Delta+\Delta')$, $\Gamma'\geq 0$ and $(Y,S+\Gamma)$ is divisorially
log terminal.  We may write
$$
(K_Y+S+\Gamma)|_S=K_S+\Phi \qquad \text{and} \qquad (K_Y+S+\Gamma+\Gamma')|_S=K_S+\Phi'.
$$
Note that the coefficients of $\Phi$ belong to $D(I)$.  Let $B$ be a prime divisor on $W$.
Let
\begin{multline*}
\mu=\sup\{\, t\in \mathbb{R} \,|\, \text{$(S,\Phi+tf^*B)$ is log canonical over a}\\
\text{neighbourhood of the generic point of $B$}\,\},
\end{multline*}
be the log canonical threshold over a neighbourhood of the generic point of $B$.  
We define $\Theta$ by 
$$
\mult_B(\Theta)=1-\mu.
$$
It is clear that the coefficients of $\Theta$ belong to 
$$
\{\, a \,|\, 1-a\in \Lct_{n-1}(D(I)) \,\}\cup \{\,1\,\}.
$$

Let
\begin{multline*}
\lambda=\sup\{\, t\in \mathbb{R} \,|\, \text{$(S,\Phi'+tf^*B)$ is log canonical over a}\\
\text{neighbourhood of the generic point of $B$}\,\},
\end{multline*}
be the log canonical threshold over a neighbourhood of the generic point of $B$.  We
define a divisor $\Theta_b$ on $W$ by
$$
\mult_B(\Theta_b)=1-\lambda.
$$
As $\Gamma'\geq 0$ we have $\Phi\leq \Phi'$, so that $\lambda\leq \mu$.  But then
$$
\Theta\leq\Theta_b.
$$
Note that $\Theta_b$ is precisely the divisor defined in Kawamata's subadjunction formula,
see Theorem 1 and 2 of \cite{Kawamata98} and also (8.5.1) and (8.6.1) of \cite{Kollar07}.
It follows that the difference
$$
(K_X+\Delta+\Delta')|_W-(K_W+\Theta_b),
$$ 
is pseudo-effective, so that the difference
$$
(K_X+\Delta+\Delta')|_W-(K_W+\Theta),
$$
is certainly pseudo-effective.  

Now suppose that $V$ is the general member of a covering family of subvarieties of $X$.
We first relate the definition of $\Theta$, which uses the log canonical threshold on $S$,
to a log canonical threshold on $X$.  Let $B$ be a prime divisor on $W$.  Pick any 
$\mathbb{Q}$-divisor $H\geq 0$ on $X$ which is $\mathbb{Q}$-Cartier in a neighbourhood of
the generic point of $B$ such that
$$
\mult_B(H|_W)=1.
$$
We have
$$
K_Y+S+\Gamma+tg^*H=g^*(K_X+\Delta+tH)+E,
$$
and so
$$
(K_Y+S+\Gamma+tg^*H)|_S=K_S+\Phi+tf^*B,
$$
over a neighbourhood of the generic point of $B$.  Now if $(X,\Delta+tH)$ is not log
canonical in a neighbourhood of the generic point of $B$ then $K_Y+S+\Gamma+tg^*H$ is not
log canonical over a neighbourhood of the generic point of $B$.  Inversion of adjunction
on $Y$, cf \cite{Kawakita05}, implies that $K_Y+S+\Gamma+tg^*H$ is log canonical over a
neighbourhood of the generic point of $B$ if and only if $K_S+\Phi+tf^*B$ is log canonical
over a neighbourhood of the generic point of $B$.  It follows that if
\begin{multline*}
\mu=\sup\{\, t\in \mathbb{R} \,|\, \text{$(S,\Phi+tf^*B)$ is log canonical over a}\\
\text{neighbourhood of the generic point of $B$}\,\},
\end{multline*}
the log canonical threshold of $f^*B$ over a neighbourhood of the generic point of $B$,
and
$$
\xi=\sup\{\, t\in \mathbb{R} \,|\, \text{$(X,\Delta+tH)$ is log canonical at the generic point of $B$}\,\},
$$
the log canonical threshold of $H$ at the generic point of $B$, then $\mu\leq \xi$.

Pick a component $\mathcal{H}$ of the Hilbert scheme whose universal family dominates $X$,
to which $V$ belongs.  By assumption $V$ is the general member of a family $R_0\subset
\mathcal{H}$ which covers $X$.  Cutting by hyperplanes in $R_0$, we may find $R\subset
R_0$ such that if $\pi\colon\map Z.R.$ is the restriction of the normalisation of the
universal family, then the natural morphism $h\colon\map Z.X.$ is generically finite.  We
may write
$$
K_Z+\Xi=h^*(K_X+\Delta).
$$
Possibly blowing up, we may assume that $(Z,\Xi)$ has simple normal crossings over a dense
open subset $R_1$ of $R$.  Let $U$ be the fibre of $\pi$ corresponding to $W$.  As $V$ is
a general member of $R$, we may assume that $r=\pi(U)\in R_1$ and so $(U,\Xi|_U)$ has
simple normal crossings.  As the coefficients of $\Xi|_U$ are at most one it follows that
$(U,\Xi|_U)$ is sub log canonical.  Therefore it is enough to check that
$$
K_U+\Psi\geq (K_X+\Delta)|_U=K_U+\Xi|_U,
$$
on the given model and in fact we just have to check that $\Psi\geq \Xi|_U$.

Let $C$ be a prime divisor on $U$.  If $\mult_C\Xi|_U\leq 0$ there is nothing to prove as
$\Psi\geq 0$.  If $C$ is an exceptional divisor of $\map U.V.$ then $\mult_C\Psi=1$ and
there is again nothing to prove as $\mult_C\Xi|_U\leq 1$.

Otherwise pick a prime component $G$ of $\Xi$ such that $\mult_C(G|_U)=1$.  If $h(G)$ is a
divisor then let $H=h(G)/e$ where $e$ is the ramification index at $G$.  Note that $H$ is
$\mathbb{Q}$-Cartier in a neighbourhood of the generic point of $B=\psi(C)$.  Otherwise,
pick a $\mathbb{Q}$-Cartier divisor $H\geq 0$, which does not contain $V$, such that
$\mult_G(h^*H)=1$.  Either way, as $r\in R$ is general it follows that
$\mult_C(h^*H|_U)=1$.  But then
$$
\mult_B(H|_W)=\mult_C(h^*H|_U)=1.
$$

We may write
$$
K_Z+\Xi+\xi h^*H=h^*(K_X+\Delta+\xi H).
$$
As $(X,\Delta+\xi H)$ is log canonical in a neighbourhood of the generic point of $B$,
$K_Z+\Xi+\xi h^*H$ is sub log canonical in a neighbourhood of the generic point of $C$.
Note that in a neighbourhood of the generic point of $C$, 
$$
(K_Z+\Xi+\xi h^*H)|_U=K_U+\Xi|_U+\xi C+J,
$$
where $J\geq 0$.  As $r$ is a general point of $R$, $(U,\Xi|_U+\xi C+J)$ is sub log
canonical in a neighbourhood of the generic point of $C$.  It follows that
$$
\mult_C \Xi|_U+\xi\leq 1,
$$
so that 
$$
\mult_C\Psi=\mult_B\Theta=1-\mu\geq 1-\xi\geq \mult_C\Xi|_U.
$$
Thus $\Psi\geq \Xi|_U$.  \end{proof}

\section{Global to local}
\label{s_local-global}

\begin{lemma}\label{l_global-to-local} Fix a positive integer $n$ and a set $1\in I\subset
[0,1]$.

Suppose $(X,\Delta)$ is a log canonical pair where $X$ is a variety of dimension $n+1$,
the coefficients of $\Delta$ belong to $I$ and there is a non kawamata log terminal centre
$V\subset X$.  Suppose that $c\in I$ is the coefficient of some component $M$ of $\Delta$
which contains $V$.

Then we may find a log canonical pair $(S,\Theta)$ where $S$ is a projective variety of
dimension at most $n$, the coefficients of $\Theta$ belong to $D(I)$, $K_S+\Theta$ is
numerically trivial and some component of $\Theta$ has coefficient
$$
\frac{m-1+f+kc}m,
$$
where $m$, $k\in \mathbb{N}$ and $f\in D(I)$.  
\end{lemma}
\begin{proof} Possibly passing to an open subset of $X$ and replacing $V$ by a maximal
(with respect to inclusion) non kawamata log terminal centre, we may assume that $X$ is
quasi-projective, $V$ is the only non kawamata log terminal centre of $(X,\Delta)$
contained in $M$ and that every component of $\Delta$ contains $V$.  If $V$ is a divisor
then $M=V$ is a component of $\Delta$ with coefficient one so that $c=1$.  As $1\in I$ we
may take $(S,\Theta)=(\pr 1.,p+q)$, where $p$ and $q$ are two points of $\pr 1.$.

Suppose $Y$ is the normalisation of a component of $\Delta$ of coefficient one.  Then we
may write
$$
(K_X+\Delta)|_Y=K_Y+\Gamma,
$$
for some divisor $\Gamma$ whose coefficients belong to $D(I)$, $(Y,\Gamma)$ is log
canonical and there is a component $N$ of $\Gamma$ with coefficient
$$
d=\frac{l-1+g+jc}l,
$$
containing $V$.  Here $j$ and $l$ are positive integers and $g\in I_+$.  Moreover, by
inversion of adjunction, cf \cite{Kawakita05}, every component of the inverse image of $V$
is a non kawamata log terminal centre of $(Y,\Gamma)$.  By induction on $n$ we may find a
log canonical pair $(S,\Theta)$ where $S$ is a projective variety of dimension at most
$n-1$, the coefficients of $\Theta$ belong to $D(D(I))=D(I)$, $K_S+\Theta$ is numerically
trivial and some component of $\Theta$ has coefficient
$$
\frac{m-1+f+kd}m,
$$
where $m$, $k\in \mathbb{N}$ and $f\in D(D(I))=D(I)$.  As the coefficients of $\Theta$ are
at most one, either $l=1$ or $k=1$.  Now
$$
\frac{l(m-1+f)+k(l-1)+kg+kjc}{lm}=\frac{lm-l+(lf+kg)+kl-k+(kj)c}{lm}.
$$
If $l=1$ then we have 
$$
\frac{m-1+(f+kg)+(kj)c}m
$$
and if $k=1$ then we have 
$$
\frac{lm-1+(lf+g)+jc}{lm}.
$$
Either way we are done. 

Otherwise we may assume that $\rfdown \Delta.=0$.  Let $f\colon\map Y.X.$ be a
divisorially log terminal model of $(X,\Delta)$.  Then $Y$ is $\mathbb{Q}$-factorial and
we may write
$$
K_Y+T+\Gamma=f^*(K_X+\Delta),
$$ 
where $\Gamma$ is the strict transform of $\Delta$, $T$ is the sum of the exceptional
divisors and the pair $(Y,T+\Gamma)$ is divisorially log terminal.  By (4) of
\eqref{p_dlt} we may choose $f$ so that $T$ contains the inverse image of $V$.  

Let $S$ be an irreducible component of $T$ which intersects the strict transform $N$ of
$M$ above the generic point of $V$.  Then we may write
$$
(K_Y+T+\Gamma)|_S=K_S+\Theta,
$$
by adjunction, where $(S,\Theta)$ is divisorially log terminal, the coefficients of
$\Theta$ belong to $D(I)$ and some component of $\Theta$ has a coefficient of the form
$$
\frac{m-1+f+kc}m,
$$
where $m$, $k\in \mathbb{N}$ and $f\in D(I)$.  As $S$ is a non kawamata log terminal
centre of $(Y,T+\Gamma)$ whose image contains $V$, the centre of $S$ on $X$ is $V$ so that
there is a morphism $\map S.V.$.  Note that $N\cap S$ dominates $V$.  If $v\in V$ is a
general point then $(S_v,\Theta_v)$ is divisorially log terminal, $S_v$ is projective of
dimension at most $n$, the coefficients of $\Theta_v$ belong to $D(I)$, some component of
$\Theta$ has a coefficient of the form
$$
\frac{m-1+f+kc}m,
$$
and $K_{S_v}+\Theta_v$ is numerically trivial.  
\end{proof}

\begin{lemma}\label{l_dcc-acc} Let $I\subset [0,1]$ be a set which satisfies the DCC.

If $J_0\subset [0,1]$ is a finite set then 
$$
I_0=\{\, c\in I \,|\, \text{$\frac{m-1+f+kc}m\in J_0$, for some $k$, $m\in \mathbb{N}$ and $f\in D(I)$} \,\}
$$
is a finite set.  
\end{lemma}
\begin{proof} We may assume that $c\neq 0$.  Suppose that 
$$
l=\frac{m-1+f+kc}m\in J_0.
$$
Then $kc\leq 1$.  As $I$ satisfies the DCC, we may find $\delta>0$ such that $c>\delta$.
It follows that $k < 1/\delta$ so that $k$ can take on only finitely many values.  As
$J_0$ is finite, we may find $\epsilon>0$ such that if $l<1$ then $l<1-\epsilon$.  But
then $m<\frac 1{\epsilon}$.  If $l=1$ then $f+kc=1$, in which case we may take $m=1$.
Either way, we may assume that $m$ takes on only finitely many values.

Fix $k$, $m$ and $l$.  Then 
$$
c=\frac {(ml-m+1)-f}k.
$$
The LHS belongs to $I$, a set which satisfies the DCC.  The RHS belongs to a set which
satisfies the ACC.  But the only set which satisfies both the DCC and the ACC is a finite
set.
\end{proof}

\begin{lemma}\label{l_D-to-A} Theorem~\ref{t_numerically-trivial}$_{n-1}$ implies 
Theorem~\ref{t_acc}$_n$.  
\end{lemma}
\begin{proof} As $I$ satisfies the DCC so does $J=D(I)$.  As we are assuming
Theorem~\ref{t_numerically-trivial}$_{n-1}$, there is a finite set $J_0\subset J$ such
that if $(S,\Theta)$ is a log canonical pair where $S$ is projective of dimension at most
$n-1$, the coefficients of $\Theta$ belong to $J$ and $K_S+\Theta$ is numerically trivial,
then the coefficients of $\Theta$ belong to $J_0$.  Let
$$
I_0=\{\, c\in I \,|\, \text{$\frac{m-1+f+kc}m\in J_0$ for some $k$ and $m\in \mathbb{N}$ and $f\in I_+$} \,\}.
$$
As $J_0$ is a finite set, \eqref{l_dcc-acc} implies that $I_0$ is also a finite set.  

Suppose that $(X,\Delta)$ is a log canonical pair where $X$ is a quasi-projective variety
of dimension $n$, the coefficients of $\Delta$ belong to $I$, and there is a non kawamata
log terminal centre $Z\subset X$ which is contained in every component of $\Delta$.
\eqref{l_global-to-local} implies that the coefficients of $\Delta$ belong to
$I_0$. \end{proof}

\section{Upper bounds on the volume}
\label{s_upper}

\begin{lemma}\label{l_epsilon} Theorem~\ref{t_numerically-trivial}$_{n-1}$ 
and Theorem~\ref{t_acc}$_{n-1}$ imply that there is a constant $\epsilon>0$ with the
following property:

If $(X,\Delta)\in \mathfrak{D}$, where $X$ has dimension $n$, $\Delta$ is big and
$K_X+\Phi$ is numerically trivial, where
$$
\Phi\geq (1-\delta)\Delta,
$$
for some $\delta<\epsilon$, then $(X,\Phi)$ is kawamata log terminal.  
\end{lemma}
\begin{proof} Theorem~\ref{t_numerically-trivial}$_{n-1}$ and Theorem~\ref{t_acc}$_{n-1}$
imply that we may find $\epsilon>0$ with the following property: if $S$ is a projective
variety of dimension $n-1$, $(S,\Theta)$ and $(S,\Theta')$ are two log pairs, the
coefficients of $\Theta$ belong to $D(I)$, and
$$
(1-\epsilon)\Theta\leq \Theta'\leq \Theta, 
$$
then $(S,\Theta)$ is log canonical if $(S,\Theta')$ is log canonical, and 
moreover $\Theta=\Theta'$ if in addition $K_S+\Theta'$ is numerically trivial.  

Suppose that $(X,\Phi)$ is not kawamata log terminal, where
$$
\Phi\geq (1-\delta)\Delta,
$$
for some $\delta<\epsilon$ and $K_X+\Phi$ is numerically trivial.  Pick $\lambda\in (0,1]$
such that $(X,(1-\lambda)\Delta+\lambda\Phi)$ is log canonical but not kawamata log
terminal.  As $\Phi$ is big, $\delta<\epsilon$ and $(X,\Delta)$ is kawamata log terminal,
perturbing $\Phi$ we may assume that $(X,(1-\lambda)\Delta+\lambda\Phi)$ is purely log
terminal and the non kawamata log terminal locus is irreducible.

Replacing $\Phi$ by $(1-\lambda)\Delta+\lambda\Phi$ we may assume that $(X,\Phi)$ is
purely log terminal.  Let $\phi\colon\map Y.X.$ be a divisorially log terminal
modification of $(X,\Phi)$.  We may write
$$
K_Y+\Psi=\phi^*(K_X+\Phi) \qquad \text{and} \qquad K_Y+\Gamma+aS=\phi^*(K_X+\Delta),
$$
where $S=\rfdown \Psi.$ is a prime divisor, $\Gamma$ is the strict transform of $\Delta$
and $a<1$, as $(X,\Delta)$ is kawamata log terminal.

As $K_Y+\Psi$ is numerically trivial, $K_Y+\Psi-S$ is not pseudo-effective.  By
\cite{BCHM06}, we may run $f\colon\rmap Y.W.$ the $(K_Y+\Psi-S)$-MMP until we end with a
Mori fibre space $\pi\colon\map W.Z.$.  As $K_Y+\Psi$ is numerically trivial, every step
of this MMP is $S$-positive, so that the strict transform $T$ of $S$ dominates $Z$.  Let
$F$ be the general fibre of $\pi$.  Replacing $Y$, $\Gamma$ and $\Psi$ by $F$ and the
restriction of $\pi_*\Gamma$ and $\pi_*\Psi$ to $F$, we may assume that $S$, $\Psi$ and
$\Gamma$ are $\mathbb{Q}$-linearly equivalent to multiples of the same ample divisor.

In particular $K_Y+\Gamma+S$ is ample.  As $\Psi\geq (1-\epsilon)\Gamma+S$, it follows
that $K_Y+(1-\eta)\Gamma+S$ is numerically trivial, for some $0<\eta<\epsilon$, and
$K_Y+(1-\epsilon)\Gamma+S$ is log canonical.  We may write
\begin{align*} 
(K_Y+(1-\epsilon)\Gamma+S)|_S&=K_S+\Theta_1,  \\ 
(K_Y+(1-\eta)\Gamma+S)|_S&=K_S+\Theta_2, \quad \text{and} \\ 
(K_Y+\Gamma+S)|_S&=K_S+\Theta,
\end{align*} 
where the coefficients of $\Theta$ belong to $D(I)$.  Note that 
$$
(1-\epsilon)\Theta\leq \Theta_1 \leq \Theta_2 \leq \Theta,
$$
where by \eqref{l_coefficients} the first inequality follows from the inequality
$$
t\left (\frac{m-1+f}m\right )\leq \frac{m-1+tf}m \qquad \text{for any} \qquad t\leq 1.
$$
As $(S,\Theta_1)$ is log canonical, it follows that $(S,\Theta)$ is log canonical.  In
particular $(S,\Theta_2)$ is also log canonical.  As $K_S+\Theta_2$ is numerically
trivial, $\Theta=\Theta_2$, a contradiction.  \end{proof}

\begin{lemma}\label{l_A-to-B} Theorem~\ref{t_numerically-trivial}$_{n-1}$ and Theorem~\ref{t_acc}$_{n-1}$ 
imply Theorem~\ref{t_upper}$_n$.
\end{lemma}
\begin{proof} Let $\epsilon>0$ be the constant given by \eqref{l_epsilon}.  If
$(X,\Delta)\in \mathfrak{D}$, $\Pi\sim_{\mathbb{R}} \eta\Delta$ and
$(X,\Pi+(1-\eta)\Delta)$ is not kawamata log terminal, then $\eta\leq \epsilon$.  It
follows that
\[
\vol(X,\Delta)\leq \left (\frac n{\epsilon}\right )^n.\qedhere
\]
\end{proof}

\section{Birational boundedness}
\label{s_birational}

\begin{lemma}\label{l_volume} Let $(X,\Delta)$ be a log pair, where $X$ is a projective variety 
of dimension $n$ and let $D$ be a big $\mathbb{R}$-divisor.

If $\vol(X,D)>(2n)^n$ then there is a family $\map V.B.$ of subvarieties of $X$ such that
if $x$ and $y$ are two general points of $X$ then we may find $b\in B$ and $0\leq
D_b\sim_{\mathbb{R}} D$ such that $(X,\Delta+D_b)$ is log canonical but not kawamata log
terminal at both $x$ and $y$ and there is a unique non kawamata log terminal place of
$(X,\Delta+D_b)$ with centre $V_b$ containing $x$.
\end{lemma}
\begin{proof} Let $K$ be the algebraic closure of the function field of $X$.  
There is a fibre square
$$
\begin{diagram}
X_K   &   \rTo    &   X \\
\dTo     &            & \dTo \\
\sp K   &   \rTo  &   \sp k.
\end{diagram}
$$

Let $\xi$ be the closed point of $X_K$ corresponding to the generic point of $X$ and let
$D_K$ be the pullback of $D$.  As $\xi$ is a smooth point of $X_K$ and
$\vol(X_K,D_K/2)>n^n$, we may find $0\leq D_{\xi} \sim_{\mathbb{R}} D_K/2$ such that
$(X_K,\Delta_K+D_{\xi})$ is not kawamata log terminal at $\xi$.  By standard arguments we
may spread out $D_{\xi}$ to a family of divisors $D_t$, $t\in T$ such that if $x$ is a
general point of $X$ then we may find $t\in T$ such that $(X,\Delta+D_t)$ is not kawamata
log terminal at $x$, where $D_t \sim_{\mathbb{Q}} D/2$.

Let $y$ be a general point of $X$.  Pick $s$ such that $(X,\Delta+D_s)$ is not kawamata
log terminal at $y$, where $D_s \sim_{\mathbb{Q}} D/2$.  Let
$$
\beta=\beta_x=\sup \{\, \lambda\in \mathbb{R} \,|\, \text{$(X,\Delta+\lambda(D_t+D_s))$ is log canonical at $x$} \,\},
$$
be the log canonical threshold.  As $x$ and $y$ are general points of $X$,
$\beta_x=\beta_y$.  Thus $(X,\Delta+\beta(D_s+D_t))$ is log canonical but not kawamata log
terminal at $x$ and $y$.  Perturbing we may assume that there is a unique non kawamata log
terminal place of $(X,\Delta+\beta(D_t+D_s))$ containing $x$ with centre $V_{(s,t)}$.
Then we may take for $B$ an open subset of $T\times T$ of the form $U\times U$.
\end{proof}

\begin{lemma}\label{l_klt-case} Assume Theorem~\ref{t_birational}$_{n-1}$
and Theorem~\ref{t_acc}$_{n-1}$.  Fix a positive integer $p$.

Let $\mathfrak{B}_1$ be the set of kawamata log terminal pairs $(X,\Delta)$, where $X$ is
projective of dimension $n$, $K_X+\Delta$ is big and either $p\Delta$ is integral or the
coefficients of $\Delta$ belong to
$$
\{\, \frac{r-1}r \,|\, r\in \mathbb{N} \,\}.
$$

Then there is a positive integer $m$ such that $\phi_{m(K_X+\Delta)}$ is birational, for 
every $(X,\Delta)\in \mathfrak{B}_1$.  
\end{lemma}
\begin{proof} Passing to a log canonical model of $(X,\Delta)$ we may assume that
$K_X+\Delta$ is ample.

Pick a positive integer $k$ such that $\vol(X,k(K_X+\Delta))>(2n)^n$.  We will apply
\eqref{t_inductive} to $k(K_X+\Delta)$.  (2) holds by hypothesis.  

Let 
$$
J=\{\, 1-a\,|\, a\in \Lct_{n-1}(D(I)) \,\}\cup \{\,1\}.
$$
Theorem~\ref{t_acc}$_{n-1}$ implies that $J$ satisfies the DCC.

Theorem~\ref{t_birational}$_{n-1}$ implies that there is a positive integer $l$ such that
if $(U,\Psi)$ is a log canonical pair, where $U$ is projective of dimension at most $n-1$,
the coefficients of $\Psi$ belong to $J$ and $K_U+\Psi$ is big, then $\phi_{l(K_U+\Psi)}$
is birational.

Apply \eqref{l_volume} to $k(K_X+\Delta)$ to get a family $\map V.B.$.  Let $b\in B$ be a
general point.  Let $\nu\colon\map W.V_b.$ be the normalisation of $V_b$.
\eqref{t_coefficients}$_n$ implies that we may find $\Theta$ on $W$ such that
$$
(K_X+\Delta+\Delta_b)|_W-(K_W+\Theta),
$$
is pseudo-effective, where the coefficients of $\Theta$ belong to $J$.  

Let $\psi\colon\map U.W.$ be a log resolution of $(W,\Theta)$ and let $\Psi$ be the sum of
the strict transform of $\Theta$ and the exceptional divisors.  \eqref{t_coefficients}$_n$
implies that
$$
(K_U+\Psi)\geq (K_X+\Delta)|_U,
$$
so that $K_U+\Psi$ is big.  As the coefficients of $\Theta$ belong to $J$, it follows that
the coefficients of $\Psi$ belong to $J$.  But then $\phi_{l(K_U+\Psi)}$ is birational.
It is easy to see (1) of \eqref{t_inductive} holds.  

As the hypotheses of \eqref{t_inductive} hold, there is a positive integer $m_0$ such that
$\phi_{m_0k(K_X+\Delta)}$ is birational.  If $\vol(X,K_X+\Delta)\geq 1$ then
$\vol(X,2(n+1)(K_X+\Delta))>(2n)^n$ and $\phi_{2m_0(n+1)(K_X+\Delta)}$ is birational.  

Otherwise, if $\vol(X,K_X+\Delta)<1$, then we may find $k$ such that
$$
(2n)^n<\vol(X,k(K_X+\Delta))\leq (4n)^n.
$$
It follows that 
$$
\vol(X,m_0k(K_X+\Delta))\leq (4m_0n)^n.
$$
\eqref{t_birationally-bounded} implies that there is a constant $0<\delta<1$ such that if
$(X,\Delta)\in \mathfrak{B}$, then
$$
\vol(X,K_X+\Delta)>\delta.
$$
In this case, 
$$
\vol(X,\alpha(K_X+\Delta))>(2n)^n,
$$
where 
$$
\alpha=\frac {2n}{\delta},
$$
and we may take $m=\max(m_0\rup \alpha.,2m_0(n+1))$.  \end{proof}

\begin{lemma}\label{l_psef} Assume Theorem~\ref{t_birational}$_{n-1}$, Theorem~\ref{t_acc}$_{n-1}$, and
Theorem~\ref{t_upper}$_n$.

Then there is a constant $\beta<1$ such that if $(X,\Delta)\in \mathfrak{B}$ then the
pseudo-effective threshold
$$
\lambda=\inf \{\, t\in \mathbb{R} \,|\, \text{$K_X+t\Delta$ is big}\,\},
$$
is at most $\beta$.  
\end{lemma}
\begin{proof} We may assume that $1\in I$.  Suppose that $(X,\Delta)\in \mathfrak{B}$.
Let $\pi\colon\map Y.X.$ be a log resolution of $(X,\Delta)$.  We may write
$$
K_Y+\Gamma=\pi^*(K_X+\Delta)+F,
$$
where $\Gamma$ is the strict transform of $\Delta$ plus the sum of the exceptional 
divisors and $F\geq 0$ is exceptional as $(X,\Delta)$ is log canonical.  Let 
$$
\mu=\inf \{\, t\in \mathbb{R} \,|\, \text{$K_Y+t\Gamma$ is big}\,\}.
$$
be the pseudo-effective threshold.  As $\pi_*(K_Y+\mu\Gamma)=K_X+\mu\Delta$ is
pseudo-effective it follows that $\lambda\leq \mu$ and so it suffices to bound $\mu$ away
from one.  Replacing $(X,\Delta)$ by $(Y,\Gamma)$ we may assume that $(X,\Delta)$ has
simple normal crossings.

We may assume that $\lambda>1/2$, so that $K_X$ is not pseudo-effective.  As $K_X+\Delta$ 
is big we may find $0\leq D \sim_{\mathbb{R}} (K_X+\Delta)$.  If $\epsilon>0$ then 
$$
(1+\epsilon)(K_X+\lambda \Delta) \sim_{\mathbb{R}} K_X+\mu\Delta+\epsilon D,
$$
where $\mu=(1+\epsilon)\lambda-\epsilon<\lambda$.  It follows that if $\epsilon$ is
sufficiently small then $K_X+\mu\Delta+\epsilon D$ is kawamata log terminal.  By
\cite{BCHM06}, we may run $f\colon\rmap X.Y.$ the $(K_X+\lambda\Delta)$-MMP with scaling
until $K_Y+\Gamma$ is kawamata log terminal and nef, where $\Gamma=f_*(\lambda \Delta)$.
Now we may run the $(K_Y+\mu f_*\Delta)$-MMP with scaling of $f_*D$ until we get to a Mori
fibre space $\pi\colon\map Y.Z.$; all steps of this MMP are $(K_Y+\Gamma)$-trivial, as all
steps of this MMP are $(K_Y+\mu f_*\Delta+\epsilon f_*D)$-trivial, so that $(Y,\Gamma)$
remains kawamata log terminal and nef.  Replacing $(X,\Delta)$ by a log resolution, we may
assume that $f$ is a morphism.  Replacing $X$ by the general fibre of the composition of
$f$ and $\pi$, we may assume that $Z$ is a point, so that $K_Y+\Gamma$ is numerically
trivial.

Suppose that we have a sequence of such log pairs $(X_l,\Delta_l)\in \mathfrak{B}$.  We
may assume that the pseudo-effective threshold is an increasing sequence,
$$
\ialist \lambda.<.,
$$
and it suffices to bound this sequence away from one.  Let
$$
J=\{\, \lambda_li \,|\, i\in I, l\in \mathbb{N} \,\}.
$$
Then $J$ satisfies the DCC, as $\lambda_l$ are an increasing sequence.  

Theorem~\ref{t_upper}$_n$ implies that there is a constant $C$ such that
$\vol(Y,\Gamma)<C$ for any $\Gamma$ whose coefficients belong to $J$.  Let $\alpha$ be the
smallest non-zero element of $J$ and let $G=G_i$ be the sum of the components of
$\Gamma=\Gamma_i$.  Let $Y=Y_i$.  Then
\begin{align*} 
\vol(Y,K_Y+G)&=\vol(Y,G-\Gamma) \\
             &\leq \vol(Y,G) \\
             &\leq \vol(Y,\frac 1{\alpha}\Gamma) \\
             &\leq \frac C{\alpha^n}.
\end{align*} 

Let $D$ be the sum of the components of $\Delta$.  Certainly $K_X+D$ is big.  We may write
$$
K_X+D=f^*(K_Y+G)+F,
$$
where $F$ is supported on the exceptional locus.  It follows that 
$$
\vol(X,K_X+D)\leq \vol(Y,K_Y+G)\leq \frac C{\alpha^n}.
$$

Given $(X_l,D_l)$ we may pick $r\in \mathbb{N}$ such that 
$$
K_{X_l}+\Theta_l=K_{X_l}+\frac{r-1}rD_l
$$
is big.  As the coefficients of $\Theta_l$ belong to 
$$
\{\, \frac{r-1}r \,|\, r\in \mathbb{N} \,\},
$$
\eqref{l_klt-case} implies that 
$$
\{\, (X_l,\Theta_l) \,|\, l\in \mathbb{N}\,\},
$$
is log birationally bounded.  But then 
$$
\{\, (X_l,\Delta_l) \,|\, l\in \mathbb{N}\,\},
$$
is log birationally bounded.  In particular, \cite[1.9]{HMX10} implies that there is a
constant $\delta>0$ such that
$$
\vol(X_l,K_{X_l}+\Delta_l)\geq \delta,
$$ 
for every $l\in \mathbb{N}$.  In this case 
$$
\delta\leq \vol(X,K_X+\Delta)\leq \vol(Y,K_Y+\frac 1{\lambda}\Gamma)=(\frac 1\lambda-1)^n\vol(Y,\Gamma)\leq (\frac 1\lambda-1)^nC,
$$
so that we may take 
\[
\beta=\frac 1{1+\left (\frac{\delta}C\right )^{1/n}}.\qedhere
\]
\end{proof}

\begin{lemma}\label{l_B-to-C} Theorem~\ref{t_birational}$_{n-1}$, Theorem~\ref{t_acc}$_{n-1}$, and
Theorem~\ref{t_upper}$_n$ imply Theorem~\ref{t_birational}$_n$.
\end{lemma}
\begin{proof} Replacing $I$ by 
$$
I\cup \{\, \frac{r-1}r \,|\, r\in \mathbb{N}\,\}\cup \{1\},
$$
we may assume that $1$ is both an accumulation point of $I$ and an element of $I$.  Let
$\alpha$ be the smallest non-zero element of $I$.  By \eqref{l_psef} there is a constant
$\beta<1$ such that if $(X,\Delta)\in \mathfrak{B}$ then the pseudo-effective threshold
$$
\lambda=\inf \{\, t\in \mathbb{R} \,|\, \text{$K_X+t\Delta$ is big}\,\},
$$
is at most $\beta$.  

Pick $(X,\Delta)\in \mathfrak{B}$.  Let $\pi\colon\map Y.X.$ be a log resolution 
of $(X,\Delta)$.  Then we may write 
$$
K_Y+\Gamma=\pi^*(K_X+\Delta)+E,
$$
where $\Gamma$ is the strict transform of $\Delta$ plus the sum of the exceptional
divisors.  Replacing $(X,\Delta)$ by $(Y,\Gamma)$ we may assume that $(X,\Delta)$ is log
smooth.  If $S=\rfdown \Delta.$, then we may pick $r\in \mathbb{N}$ such that
$$
K_X+\Delta'=K_X+\frac{r-1}r S+\{\Delta\},
$$
is big.  Replacing $(X,\Delta)$ by $(X,\Delta')$, we may assume that $(X,\Delta)$ is
kawamata log terminal.

Pick $p$ such that
$$
p>\frac 2{\alpha(1-\beta)}.
$$
If $a$ is the coefficient of a component of $\Delta$ then 
\begin{align*} 
\frac{\rfdown pa.}p&> a-\frac 1p\\
                   &>a-\frac{\alpha(1-\beta)}2\\
                   &\geq a-\frac{a(1-\beta)}2\\
                   &=\frac{a(1+\beta)}2.
\end{align*} 
It follows that 
$$
\frac{\beta+1}2 \Delta\leq \bdd \Delta.p.\leq \Delta,
$$
so that $K_X+\bdd \Delta.p.$ is big.  Since the coefficients of
$\bdd\Delta.p.$ belong to
$$
I_0=\{\, \frac ip \,|\, 1\leq i\leq p-1 \,\},
$$
\eqref{l_klt-case} implies that there is a positive integer $m$ such that
$\phi_{m(K_X+\bdd \Delta.p.)}$ is birational.  But then $\phi_{m(K_X+\Delta)}$ is
birational as well.  \end{proof}
\section{Numerically trivial log pairs}
\label{s_trivial}

\begin{lemma}\label{l_C-to-D} Theorem~\ref{t_numerically-trivial}$_{n-1}$ 
and Theorem~\ref{t_birational}$_n$ imply Theorem~\ref{t_numerically-trivial}$_n$.
\end{lemma}
\begin{proof} We may assume that $1\in I$ and $n>1$.  

As we are assuming Theorem~\ref{t_numerically-trivial}$_{n-1}$ there is a finite set
$J_0\subset J=D(I)$ with the following property.  If $(S,\Theta)$ is a log pair such that
$S$ is projective of dimension $n-1$, the coefficients of $\Theta$ belong to $J$,
$(S,\Theta)$ is log canonical, and $K_S+\Theta$ is numerically trivial, then the
coefficients of $\Theta$ belong to $J_0$.  Let $I_1$ be the largest subset of $I$ such
that $D(I_1)\subset J_0$.  \eqref{l_dcc-acc} implies that $I_1$ is finite.  

Theorem~\ref{t_birational}$_n$ implies that there is a constant $m$ with the following
property: if $(Y,\Gamma)$ is log canonical, $Y$ is a projective variety of dimension $n$,
$K_Y+\Gamma$ is big and the coefficients of $\Gamma$ belong to $I$, then
$\phi_{m(K_Y+\Gamma)}$ is birational.

For every $1\leq l\leq m$, let 
$$
A_l=[(l-1)/m,l/m),
$$
and $A_{m+1}=\{\,1\,\}$ so that
$$
[0,1]=\bigcup_{l=1}^{m+1} A_l.
$$
Let $I_2$ be the union of the largest elements of $A_l\cap I$ (if $A_l\cap I$ does not
have a largest element, either because it is empty or because it has infinitely many
elements, then we ignore the elements of $A_l\cap I$).  Then $I_2$ has at most $m+1$
elements, so that $I_2$ is certainly finite.  Let $I_0$ be the union of $I_1$ and $I_2$.

Suppose that $(X,\Delta)$ satisfies (1--4) of Theorem~\ref{t_numerically-trivial}$_n$.
Passing to a divisorially log terminal modification, we may assume that $X$ is
$\mathbb{Q}$-factorial.  Further $(X,\Delta)$ is kawamata log terminal if and only if
$\rfdown\Delta.=0$.  Suppose that $B$ is a prime component of $\Delta$ with coefficient
$i$.  It suffices to prove that $i\in I_0$.  We may assume that $i\neq 1$.  Suppose that
$B$ intersects a component of $\rfdown \Delta.$.  If $S$ is the normalisation of this
component then by adjunction we may write
$$
(K_X+\Delta)|_S=K_S+\Theta,
$$
where the coefficients of $\Theta$ belong to $J=D(I)$ by \eqref{l_coefficients}.  As $S$
is projective of dimension $n-1$, $(S,\Theta)$ is log canonical, and $K_S+\Theta$ is
numerically trivial, the coefficients of $\Theta$ belong to $J_0$.  But then $i\in I_1$.

As $K_X+\Delta$ is numerically trivial, $K_X+\Delta-iB$ is not pseudo-effective.  By
\cite{BCHM06} we may run $f\colon\rmap X.Y.$ the $(K_X+\Delta-iB)$-MMP until we reach a
Mori fibre space.  As $K_X+\Delta$ is numerically trivial, it follows that every step of
this MMP is $B$-positive.  If at some step of this MMP we contract a component $S$ of
$\rfdown \Delta.$ then this component intersects $B$ and $i\in I_1$ by the argument above.
It follows that $(Y,f_*\Delta)$ is kawamata log terminal if and only if $\rfdown
f_*\Delta.=0$.  Further $B$ is not contracted and so replacing $(X,\Delta)$ by
$(Y,f_*\Delta)$, we may assume that $X$ is a Mori fibre space $\pi\colon\map X.Z.$, where
$B$ dominates $Z$.

If $Z$ is not a point, then replacing $X$ by the general fibre of $\pi$ we are done by
induction.  So we may assume that $X$ has Picard number one.  If $\rfdown \Delta.\neq 0$
then any component $S$ of $\rfdown \Delta.$ intersects $B$ and so $i\in I_1$.  Otherwise
$\rfdown \Delta.=0$ and we may assume that $(X,\Delta)$ is kawamata log terminal.

Suppose that $j\in I$ and $j>i$.  Let $\pi\colon\map Y.X.$ be a log resolution of
$(X,\Delta)$.  Let $\Gamma_0$ be the strict transform of $\Delta$, let $E$ by the sum of
the exceptional divisors, and let $C$ be the strict transform of $B$.  Set
$$
\Gamma=\Gamma_0+E+(j-i)C.
$$
Then $(Y,\Gamma)$ is log canonical and the coefficients of $\Gamma$ belong to $I$.  We may
write
$$
K_Y+\Gamma_0+E=\pi^*(K_X+\Delta)+F,
$$
where $F\geq 0$ contains the full exceptional locus.  Pick $\epsilon>0$ such that $F\geq
\epsilon E$.  Then
$$
K_Y+\Gamma=(K_Y+\Gamma_0+(1-\epsilon)E)+(j-i)C+\epsilon E,
$$
is big.  Hence $\phi_{m(K_Y+\Gamma)}$ is birational, so that $K_Y+\bdd \Gamma.m.$ is big.
But then $K_X+\bdd \Lambda.m.$ is big, where
$$
\Lambda=\pi_*\Gamma=\Delta+(j-i)B.
$$
It follows that if $i\in A_l$, then $j\geq l/m$, so that $i$ is the largest element of the
interval $A_l$ which also belongs to $I$.  Hence $i\in I_2$.
\end{proof}
\section{Proof of Theorems}
\label{s_theorems}

\begin{proof}[Proof of \eqref{t_num} and \eqref{t_simple}] This is Theorem~\ref{t_acc}
and Theorem~\ref{t_numerically-trivial}.  
\end{proof}

\begin{proof}[Proof of \eqref{t_lct}] Suppose that $\ilist c.\in \Lct_n(I,J)$, where
$c_i\leq c_{i+1}$.  It suffices to show that $c_i=c_{i+1}$ for $i$ sufficiently large.  By
assumption we may find log canonical pairs $(X_i,\Delta_i)$ and $\mathbb{R}$-Cartier
divisors $M_i$, where $X_i$ is a variety of dimension $n$, the coefficients of $\Delta_i$
belong to $I$, the coefficients of $M_i$ belong to $J$ and $c_i$ is the log canonical
threshold,
$$
c_i=\sup \{\, t\in \mathbb{R} \,|\, \text{$(X_i,\Delta_i+c_iM_i)$ is log canonical} \,\}.
$$
Let $\Theta_i=\Delta_i+c_iM_i$ and 
$$
K=I\cup \{\, c_ij \,|\, i\in \mathbb{N}, j\in J \,\}.
$$
Then $(X_i,\Theta_i)$ is log canonical, $X_i$ is a variety of dimension $n$, the
coefficients of $\Theta_i$ belong to $K$ and there is a non kawamata log terminal centre
$V$ contained in the support of $M_i$.  Possibly throwing away components of $\Theta_i$
which don't contain $V$ and passing to an open subset which contains the generic point of
$V$, we may assume that every component of $\Theta_i$ contains $V$.

As $K$ satisfies the DCC, \eqref{t_num} implies that the coefficients of $\Theta_i$
belong to a finite subset $K_0$ of $K$.  It follows $c_i=c_{i+1}$ for $i$ sufficiently
large.
\end{proof}

\begin{proof}[Proof of \eqref{t_volume}] (3) is Theorem~\ref{t_birational}.  

Fix a constant $V>0$ and let
$$
\mathfrak{D}_V=\{\, (X,\Delta)\in \mathfrak{D} \,|\, 0<\vol(X,K_X+\Delta)\leq V \,\}.
$$
(3) implies that $\phi_{m(K_X+\Delta)}$ is birational.  \eqref{t_birationally-bounded} implies that
the set
$$
\{\, \vol(X,K_X+\Delta) \,|\, (X,\Delta)\in \mathfrak{D}_V \,\},
$$
satisfies the DCC, which implies that (1) and (2) of \eqref{t_volume} hold 
in dimension $n$.   
\end{proof}

\begin{lemma}\label{l_lc-model} Let $\map Z.T.$ be a projective morphism to a variety 
and suppose that $(Z,\Phi)$ has simple normal crossings over $T$.  Suppose that the
restriction of any irreducible component of $\Phi$ to any fibre is irreducible.  Suppose
that $(Z,\Phi)$ is kawamata log terminal and there is a closed point $0\in T$ such that
$K_{Z_0}+\Phi_0$ is big.  Let $\Theta\leq \Phi$ be any divisor with the same support as
$\Phi$.

Then we may find finitely many birational contractions $f_i\colon\rmap Z.X_i.$ over $T$ such
that if $f\colon\rmap Z_t.Y.$ is the log canonical model of $(Z_t,\Psi)$ for some $t\in T$
and $\Theta_t\leq \Psi\leq \Phi_t$ then $f=f_{it}$ for some index $i$.
\end{lemma}
\begin{proof} \cite[1.7]{HMX10} implies that $K_Z+\Phi$ is big over $T$.  
Pick 
$$
0\leq D \sim_{\mathbb{R},T}(K_Z+\Phi).
$$
Let 
$$
B=\frac{\epsilon}{1-\epsilon}D.  
$$
If we pick $\epsilon>0$ sufficiently small then $K_Z+B+\Phi$ is kawamata log terminal and
we may find a divisor $0\leq \Theta'\leq \Theta$ with
$$
K_Z+\Theta=\epsilon(K_Z+\Phi)+(1-\epsilon)(K_Z+\Theta').
$$
If $\Theta\leq \Xi\leq \Phi$ then
$$
K_Z+\Xi\sim_{\mathbb{R},T}(1-\epsilon)(K_Z+B+\Xi'),
$$
where $\Theta'\leq \Xi'\leq \Xi$.  It is proved in \cite{BCHM06} that there are finitely
many $\llist f.k.$ birational contractions $f_i\colon\rmap Z.X_i.$ over $T$ such that if
$g\colon\rmap Z.X.$ is the log canonical model of $K_Z+\Xi$ over $T$ then $g=f_i$ for some
index $1\leq i\leq k$.

It suffices to show that if $\Xi|_{Z_t}=\Psi$ and $g$ is the log canonical model of
$K_Z+\Xi$ then $f=g_t$.  For this we may assume that $T$ is affine.

In this case the (relative) log canonical model is given by taking Proj
$$
X_i=\Proj(Z,R(Z,k(K_Z+\Xi))),
$$
of the (truncation of the) canonical ring
$$
R(Z,k(K_Z+\Xi))=\bigoplus_{m\in \mathbb{N}} H^0(Z,\ring Z.(mk(K_Z+\Xi))).
$$
On the other hand \cite[1.7]{HMX10} implies that if $k$ is sufficiently divisible then
$$
\map R(Z,k(K_Z+\Xi)).R(Z_t,k(K_{Z_t}+\Psi)).,
$$
is surjective and so $f=g_t$.  
\end{proof}

\begin{proof}[Proof of \eqref{t_ample-bounded}] By definition there is a log pair $(Z,B)$
and a projective morphism $\map Z.T.$, where $T$ is of finite type with the following
property.  If $(X,\Delta)\in \mathfrak{D}$ then there is a closed point $t\in T$ and a
birational map $f\colon\rmap X.Z_t.$ such that the support of $B_t$ contains the support
of the strict transform of $\Delta$ and any $f^{-1}$-exceptional divisor.

We may assume that $T$ is reduced.  Blowing up and decomposing $T$ into a finite union of
locally closed subsets, we may assume that $(Z,B)$ has simple normal crossings; passing to
an open subset of $T$, we may assume that the fibres of $\map Z.T.$ are log pairs, so that
$(Z,B)$ has simple normal crossings over $T$; passing to a finite cover of $T$, we may
assume that every stratum of $(Z,B)$ has irreducible fibres over $T$; decomposing $T$ into
a finite union of locally closed subsets, we may assume that $T$ is smooth; finally
passing to a connected component of $T$, we may assume that $T$ is integral.

Let $a=1-\epsilon<1$.  Let $\Phi=aB$ and $\Theta=\delta B$, so that $\Phi$, $\Theta$ and
$B$ have the same support but the coefficients of $\Phi$ are all $a$, the coefficients of
$\Theta$ are all $\delta$ and the coefficients of $B$ are all one.  As $(Z,\Phi)$ is
kawamata log terminal it follows that there are only finitely many valuations of log
discrepancy at most one with respect to $(Z,\Phi)$.  As $(Z,\Phi)$ has simple normal
crossings there is a sequence of blow ups $\map Y.Z.$ of strata, which extracts every
divisor of log discrepancy at most one.  Note that as $(Z,\Phi)$ has simple normal
crossings over $T$, it follows that if $t\in T$ is a closed point then every valuation of
log discrepancy at most one with respect to $(Z_t,\Phi_t)$ has centre a divisor on $Y_t$.

Suppose that $(X,\Delta)\in \mathfrak{D}$.  Then there is a closed point $t\in T$ and a
birational map $f\colon\rmap X.Z_t.$ such that the support of $B_t$ contains the support
of the strict transform of $\Delta_t$ and any $f^{-1}$-exceptional divisor.  Let
$p\colon\map W.X.$ and $q\colon\map W.Z_t.$ resolve $f$.  Let $S$ be the sum of the
$p$-exceptional divisors and let $\Xi$ be the sum of the strict transform of $\Delta$ and
$aS$, so that $S$ and $\Xi$ are divisors on $W$.  We may write
$$
K_W+\Xi=p^*(K_X+\Delta)+E,
$$
where $E$ is a sum of $p$-exceptional divisors and $E\geq 0$ as the log discrepancy of
$(X,\Delta)$ is greater than $\epsilon$.

Let $\Psi=q_*\Xi$.  We may write
$$
p^*(K_X+\Delta)+E+F=q^*(K_{Z_t}+\Psi),
$$
where $F$ is $q$-exceptional.  As $p^*(K_X+\Delta)$ is nef, it is $q$-nef so that $E+F\geq
0$ by negativity of contraction.  If $\nu$ is any valuation whose centre is a divisor on
$X$ then
\begin{align*} 
a(Z_t,\Phi_t,\nu)& \leq a(Z_t,\Psi,\nu)  && \text{as $\Phi_t\geq \Psi$} \\
                 & \leq a(X,\Delta,\nu)  && \text{as $E+F\geq 0$} \\
                 & \leq 1                && \text{as the centre of $\nu$ is a divisor on $X$.}
\end{align*} 

Therefore the induced birational map $\rmap Y_t.X.$ is a birational contraction.  Thus
replacing $Z$ by $Y$ and $B$ by its strict transform union the exceptional divisor, we may
assume that $g=f^{-1}\colon\rmap Z_t.X.$ is a birational contraction.  In this case $F$ is
$p$-exceptional and so $g$ is the log canonical model of $(Z_t,\Theta_t)$.

Since there are only finitely integral divisors $0\leq B'\leq B$, replacing $B$ we may
assume that $\Psi$ has the same support as $B_t$.  $K_{Z_t}+\Phi_t$ is big as
$K_{Z_t}+\Psi$ is big and $\Phi_t\geq \Psi$.  Finally $\Theta_t\leq\Psi\leq \Phi_t$ and so
we are done by \eqref{l_lc-model}.
\end{proof}

\section{Proof of Corollaries}
\label{s_corrollaries}

\begin{proof}[Proof of \eqref{c_ter}] This follows from \eqref{t_lct} and the main
result of \cite{Birkar05}.
\end{proof}

\begin{proof}[Proof of \eqref{c_strong-bounded}] \eqref{t_num} implies that there is a
finite subset $I_0\subset I$ such that the coefficients of $\Delta$ belong to $I_0$.  Thus
there is a positive integer $r$ such that $r\Delta$ is integral.  

On the other hand, Theorem~\ref{t_upper} implies that there is a constant $C$ such that
$\vol(X,\Delta)<C$.  Let $D$ be the sum of the components of $\Delta$.  Then $K_X+D$ is
big and
\begin{align*} 
\vol(X,K_X+D)&=\vol(X,D-\Delta) \\
             &\leq \vol(X,D) \\
             &\leq \vol(X,r\Delta) \\
             &\leq Cr^n.
\end{align*} 

Let $\pi\colon\map Y.X.$ be a log resolution of $(X,\Delta)$.  Let $G$ be the sum of the
strict transform of the components of $\Delta$ and the exceptional divisors.  Then $(Y,G)$
has simple normal crossings.  Pick $\eta>0$ such that $(X,(1+\eta)\Delta)$ is kawamata log
terminal and the log discrepancy is greater than $\epsilon$.  Then $K_X+(1+\eta)\Delta$ is
ample and we may write
$$
K_Y+\Gamma=\pi^*(K_X+(1+\eta)\Delta),
$$
where $\Gamma\leq G$.  As $K_Y+\Gamma$ is big it follows that $K_Y+G$ is big.
\eqref{t_volume} implies that there is a positive integer $m$ such that $\phi_{m(K_Y+G)}$
is birational, for every $(X,\Delta)\in \mathfrak{D}$.  But then $\mathfrak{D}$ is log
birationally bounded by \cite[2.4.2.3-4]{HMX10}.  Now apply \eqref{t_ample-bounded}.
\end{proof}

\begin{proof}[Proof of \eqref{c_batyrev}] Let $D=-r(K_X+\Delta)$.  Then $D$ is an ample
Cartier divisor and $D-(K_X+\Delta)$ is ample.  By Koll\'ar's effective base point free
theorem (cf. \cite{Kollar93b}), there is a fixed positive integer $m$ such that the linear
system $|mD|$ is base point free.  Pick a general divisor $H\in |mD|$.  Then
$(X,\Lambda=\Delta+\frac 1{mr} H)$ is kawamata log terminal and
$$
K_X+\Lambda\sim_{\mathbb{Q}} 0.
$$

Note the coefficients of $\Lambda$ belong to the finite set
$$
I=\{\, \frac ir \,|\, 1\leq i\leq r-1\,\}\cup \{\, \frac 1{mr}\,\}.
$$
There are two ways to proceed.  On the one hand we may apply \eqref{c_strong-bounded}.

Here is a more direct approach.  Theorem~\ref{t_upper} implies that
$$
\vol(X,\Lambda), 
$$
is bounded from above.  But then
$$
\vol(X,mD)\leq (mr)^n\vol(X,\Lambda),
$$
is bounded from above.  
\end{proof}

\begin{proof}[Proof of \eqref{c_fano-index}] Suppose that $r_1\leq r_2\leq \dots$ is a
non-decreasing sequence in $R$.  For each $i$ we may find $(X,\Delta)=(X_i,\Delta_i)\in
\mathfrak{D}$ and a Cartier divisor $H$ such that $-(K_X+\Delta) \sim_{\mathbb{R}} rH$.
By the cone theorem we may find a curve $C$ such that $-(K_X+\Delta)\cdot C\leq 2n$, cf.
Theorem 18.2 of \cite{Fujino09}.  In particular $r\leq 2n$, as $H\cdot C\geq 1$.  By
Fujino's extension, \cite{Fujino09a}, of Koll\'ar's effective base point free theorem,
\cite{Kollar93b}, to the case of log canonical pairs, there is a fixed positive integer
$m$ such that the linear system $|mH|$ is base point free.  Possibly replacing $m$ by a
multiple we may assume that $m>2n$.  Pick a general divisor $D\in |mH|$.

Then $(X,\Lambda=\Delta+\frac rm D)$ is log canonical and
$$
K_X+\Lambda\sim_{\mathbb{R}} 0.
$$
Then the coefficients of $\Lambda_i=\Lambda$ belong to the set
$$
I\cup \{\, \frac{r_i}m \,|\, i\in \mathbb{N} \,\},
$$
which satisfies the DCC.  \eqref{t_simple} implies that the coefficients of $\Lambda$
belong to a finite subset.  But then $r_i=r_{i+1}$ is eventually constant and so $R$
satisfies the ACC.
\end{proof}

\section{Accumulation points}
\label{s_accumulation}

\begin{definition}\label{d_associatedset} Given $I\subset [0,1]$ and $c\in [0,1]$ let
$$
D_c(I)=\{\, a\leq 1\,|\, a=\frac{m-1+f+kc} m, k, m\in \mathbb{N},f\in I_+\,\}\subset D(I\cup\{\,c\,\}).
$$
Let $\mathfrak{N}_n(I,c)$ be the set of log canonical pairs $(X,\Delta)$ such that $X$ is
a projective variety of dimension $n$, $K_X+\Delta$ is numerically trivial and we may
write $\Delta=B+C$, where the coefficients of $B$ belong to $D(I)$ and the coefficients of
$C\neq 0$ belong to $D_c(I)$.

Let 
$$
N_n(I)=\{\, c\in [0,1] \,|\, \text{$\mathfrak{N}_n(I,c)$ is non-empty}\,\}.
$$
\end{definition} 

\begin{lemma}\label{l_simple} Let $n\in \mathbb{N}$ and $I\subset [0,1]$. 
\begin{enumerate}
\item $\Lct_n(I)\subset \Lct_{n+1}(I)$.
\item $N_n(I)\subset N_{n+1}(I)$.
\item If $f\in I_+$ and $k\in \mathbb{N}$ then 
$$
c=\frac{1-f}k\in N_n(I).
$$
\end{enumerate} 
\end{lemma}
\begin{proof} Let $E$ be an elliptic curve.  If $(X,\Delta=\sum d_i\Delta_i)$ is a log
pair then $(Y,\Gamma)$ is a log pair, where $Y=X\times E$ and $\Gamma=\sum
d_i(\Delta_i\times E)$.  By construction $\Gamma$ has the same coefficients as $\Delta$.

Note that $(X,\Delta)$ is log canonical if and only if $(Y,\Gamma)$ is log canonical.
This gives (1).  Further if $c\in [0,1]$ and $(X,\Delta)\in \mathfrak{N}_n(I,c)$ then
$(Y,\Gamma)\in \mathfrak{N}_{n+1}(I,c)$.  This is (2).

Using (2), it suffices to prove (3) when $n=1$.  Let $X=\pr 1.$ and $\Delta=B+C$, where
$B=fp+fq$, $C=2kc r$, and $p$, $q$ and $r$ are three points of $\pr 1.$.  Then
$(X,\Delta)\in \mathfrak{N}_1(I,c)$ (take $m=1$) so that $c\in N_1(I)$.  This is (3).
\end{proof}

For technical reasons, it is convenient to introduce a smaller set than 
$\mathfrak{N}_n(I,c)$:
\begin{definition}\label{d_klt-version} Given $I\subset [0,1]$ and $c\in [0,1]$ let 
$\mathfrak{K}_n(I,c)\subset \mathfrak{N}_n(I,c)$ be the subset consisting of kawamata
log terminal pairs $(X,\Delta)$, where $X$ is $\mathbb{Q}$-factorial of Picard number one.

Let 
$$
K_n(I)=\{\, c\in [0,1] \,|\, \text{$\mathfrak{K}_m(I,c)$ is non-empty, for some $m\leq n$}\,\}.
$$
\end{definition}

\begin{lemma}\label{l_complex} If $n\in \mathbb{N}$ and $I\subset [0,1]$ then 
$$
N_n(I\cup \{\,1\,\})=K_n(I).
$$
In particular, $N_n(I\cup \{\,1\,\})=N_n(I)$.  
\end{lemma}
\begin{proof} By (2) of \eqref{l_simple}, it suffices to show that
$$
N_n(I\cup \{\,1\,\})\subset K_n(I).
$$
Suppose that $c\in N_n(I\cup \{\,1\,\})$.  Then we may find $(X,\Delta)\in
\mathfrak{N}_n(I\cup \{\,1\,\},c)$.  By assumption we may write $\Delta=A+B+C$, where the
coefficients of $A$ are one, the coefficients of $B$ belong to $D(I)$ and the coefficients
of $C\neq 0$ belong to $D_c(I)$.

Let $\pi\colon\map X'.X.$ be a divisorially log terminal modification of $(X,\Delta)$.  If
we write
$$
K_{X'}+\Delta'=\pi^*(K_X+\Delta),
$$
then $X'$ is projective of dimension $n$, $X'$ is $\mathbb{Q}$-factorial, $(X',\Delta')$
is divisorially log terminal and $K_{X'}+\Delta'$ is numerically trivial.  Let $B'$ and
$C'$ be the strict transforms of $B$ and $C$ and let $A'=\Delta'-B'-C'$.  Then the
coefficients of $A'$ are one, the coefficients of $B'$ belong to $D(I)$ and the
coefficients of $C'\neq 0$ belong to $D_c(I)$.  Thus $(X',\Delta')\in \mathfrak{N}_n(I\cup
\{\,1\,\},c)$.  Replacing $(X,\Delta)$ by $(X',\Delta')$ we may assume that $X$ is
$\mathbb{Q}$-factorial and $(X,A+B)$ is divisorially log terminal.  Note that $(X,\Delta)$
is kawamata log terminal if and only if $A=0$.

Suppose that $A$ and $C$ intersect.  Let $S$ be an irreducible component of $A$ which
intersects $C$.  Then we may write
$$
(K_X+\Delta)|_S=K_S+\Theta,
$$
by adjunction, where $(S,\Theta)$ is divisorially log terminal and moreover we may write
$\Theta=A'+B'+C'$, where the coefficients of $A'$ are one, the coefficients of $B'$ belong
to $D(I)$ and the coefficients of $C'\neq 0$ belong to $D_c(I)$.  Thus $(S,\Theta)\in
\mathfrak{N}_{n-1}(I\cup \{\,1\,\},c)$.  Hence $c\in N_{n-1}(I\cup \{\,1\,\})$ and so
$c\in K_{n-1}(I)\subset K_n(I)$, by induction on $n$.

Let $f\colon\rmap X.X'.$ be a step of the $(K_X+A+B)$-MMP.  As $K_X+\Delta$ is numerically
trivial, $f$ is automatically $C$-positive.  Suppose that $f$ is birational.  Let
$A'=f_*A$, $B'=f_*B$ and $C'=f_*C$, so that $\Delta'=f_*\Delta=A'+B'+C'$.  $C'\neq 0$, as
$f$ is $C$-positive.  $X'$ is a projective variety of dimension $n$, $(X',\Delta')$ is log
canonical, $K_{X'}+\Delta'$ is numerically trivial, the coefficients of $A'$ are all one,
the coefficients of $B'$ belong to $D(I)$ and the coefficients of $C'\neq 0$ belong to
$D_c(I)$.  Thus $(X',\Delta')\in \mathfrak{N}_n(I\cup \{\,1\,\},c)$.  Further $X'$ is
$\mathbb{Q}$-factorial and $(X',A'+B')$ is divisorially log terminal.  If a component of
$A$ is contracted then $A$ and $C$ intersect and we are done.  Otherwise $(X',\Delta')$ is
kawamata log terminal if and only if $A'=0$.

If we run the $(K_X+A+B)$-MMP with scaling of an ample divisor then we end with a Mori
fibre space.  Therefore, replacing $(X,\Delta)$ by $(X',\Delta')$ finitely many times, we may assume
that $f\colon\rmap X.Z=X'.$ is a Mori fibre space and $C$ dominates $Z$.  If $\dim Z>0$
then let $z\in Z$ be a general point.  Then $(X_z,\Delta_z)\in
\mathfrak{N}_{n-k}(I\cup\{\,1\,\},c)$, where $k=\dim Z$, and we are done by induction on
the dimension.

So we may assume that $Z$ is a point in which case $X$ has Picard number one.  If $A\neq
0$ then $A$ and $C$ intersect and we are done.  If $A=0$ then $(X,\Delta)$ is kawamata log
terminal and so $(X,\Delta)\in \mathfrak{K}_n(I,c)$.  But then $c\in K_n(I)$.  
\end{proof}

\begin{proposition}\label{p_local-global} If $I\subset [0,1]$, $I=I_+$ and $n\in \mathbb{N}$ then
$\Lct_{n+1}(I)=N_n(I)$.
\end{proposition}
\begin{proof} We first show that $\Lct_{n+1}(I)\subset N_n(I)$.  Pick $0\neq c\in
\Lct_{n+1}(I)$.  By definition we may find a log canonical pair $(X,\Delta+cM)$ where $X$
has dimension $n+1$, the coefficients of $\Delta$ belong to $I$, $M$ is an integral
$\mathbb{Q}$-Cartier divisor and there is a non kawamata log terminal centre $V$ contained
in the support of $M$.  Possibly passing to an open subset of $X$ and replacing $V$ by a
maximal non kawamata log terminal centre, we may assume that $V$ is the only non kawamata
log terminal centre of $(X,\Delta+cM)$.  In particular, $(X,\Delta)$ is kawamata log
terminal.

If $V$ is a component of $M$ then $V$ has coefficient one in $\Delta+cM$ and $c=\frac
{1-f}k\in N_n(I)$ by (3) of \eqref{l_simple}.  Otherwise let $f\colon\map Y.X.$ be a
divisorially log terminal model of $(X,\Delta+cM)$.  Then $Y$ is $\mathbb{Q}$-factorial
and we may write
$$
K_Y+T+\Delta'+cM'=f^*(K_X+\Delta+cM)
$$ 
where $\Delta'$ and $M'$ are the strict transforms of $\Delta$ and $M$, $T$ is the sum of
the exceptional divisors and the pair $(Y,T+\Delta'+cM')$ is divisorially log terminal.
By (4) of \eqref{p_dlt} we may choose $f$ so that $T$ contains the inverse image of $V$.
Let $S$ be an irreducible component of $T$ which intersects $M'$.  Then we may write
$$
(K_Y+T+\Delta'+cM')|_S=K_S+\Theta,
$$
by adjunction, where $(S,\Theta)$ is divisorially log terminal and moreover we may write
$\Theta=A+B+C$, where the coefficients of $A$ are one, the coefficients of $B$ belong to
$D(I)$ and the coefficients of $C\neq 0$ belong to $D_c(I)$.  As $S$ is a non kawamata log
terminal centre, the centre of $S$ on $X$ is $V$ so that there is a morphism $\map S.V.$.
If $v\in V$ is a general point then $(S_v,\Theta_v)\in \mathfrak{N}_k(I\cup \{\,1\,\},c)$,
for some $k\leq n$.  Thus $c\in N_k(I\cup \{\,1\,\})\subset N_n(I)$.

We now show that $\Lct_{n+1}(I)\supset N_n(I)$.  Pick $0\neq c\in N_n(I)$.  Then we may
find a pair $(X,\Delta)\in \mathfrak{K}_m(I,c)$, some $m\leq n$.  If $m<n$ then we are
done by induction on the dimension.  Otherwise $X$ has dimension $n$.  As $-K_X$ is ample,
we may pick $d$ such that $-dK_X$ is very ample and embed $X$ into projective space by the
linear system $|-dK_X|$.  

Let $Y$ be the cone over $X$ and let $\Gamma_j$ be the cone over $\Delta_j$.  Then $Y$ is
a quasi-projective variety of dimension $n+1$.  $Y$ is $\mathbb{Q}$-factorial as $X$ has
Picard number one.  $(Y,\Gamma=\sum d_i\Gamma_i)$ is log canonical but not kawamata log
terminal at the vertex $p$ of the cone.  By assumption we may write
$$
d_i=\frac{m_i-1+f_i+k_ic}{m_i},
$$
for each $i$, where $m_i$ is a positive integer, $k_i$ is a non-negative integer ($k_i=0$
if $\Gamma_i$ is a component of $B_i$ and $k_i>0$ if $\Gamma_i$ is a component of $C_i$)
and $f_i\in I_+$.  Since we are working locally around $p$, the vertex of $Y$, we may find
a cover of $\pi\colon\map \tilde Y.Y.$ which ramifies over $\Gamma_i$ to index $m_i$ for
every $i$ and is otherwise unramified at the generic point of any divisor.  We may write
$$
K_{\tilde Y}+\tilde {\Gamma}=\pi^*(K_Y+\Gamma),
$$
where the coefficients of $\tilde \Gamma$ belong to the set
$$
\{\, f_i+k_ic \,|\, i \,\}.
$$
$\tilde Y$ is a $\mathbb{Q}$-factorial quasi-projective variety of dimension $n+1$ and
$(\tilde Y,\tilde \Gamma)$ is log canonical but not kawamata log terminal over any point
$q$ lying over $p$.  Let
$$
\Theta=\sum f_i\Gamma_i \qquad \text{and} \qquad M_i=\sum k_i\Gamma_i.
$$
Then the coefficients of $\Theta$ belong to $I_+=I$, $M_i$ is an integral
$\mathbb{Q}$-Cartier divisor and
$$
c=\sup \{\, t\in \mathbb{R} \,|\, \text{$(X,\Theta+tM)$ is log canonical} \,\},
$$
is the log canonical threshold.  But then $c\in \Lct_{n+1}(I)$.
\end{proof}

\begin{lemma}\label{l_number} Let $(X,\Delta)$ be a log canonical pair, where $X$ is
$\mathbb{Q}$-factorial of dimension $n$ and Picard number one and $K_X+\Delta$ 
is numerically trivial.  

 If the coefficients of $\Delta$ are at least $\delta$ then $\Delta$ has 
at most $\frac{n+1}{\delta}$ components. 
\end{lemma}
\begin{proof} \cite[18.24]{Kollaretal} implies that the sum of the coefficients of
$\Delta$ is at most $n+1$.  
\end{proof}

\begin{proposition}\label{p_stronger} Fix a positive integer $n$ and a set $I\subset
[0,1]$ whose only accumulation point is one such that $I=I_+$.

Let $\ilist c.\in [0,1]$ be a strictly decreasing sequence with limit $c\neq 0$ with the
following property.  There is a sequence of log canonical pairs $(X_i,\Delta_i)$ such that
$X_i$ is a projective variety of dimension $n$, $K_{X_i}+\Delta_i$ is numerically trivial
and we may write $\Delta_i=A_i+B_i+C_i$, where the coefficients of $A_i$ are approaching
one, the coefficients of $B_i$ belong to $D(I)$ and the coefficients of $C_i\neq 0$ belong
to $D_{c_i}(I)$.

Then $c\in N_{n-1}(I)$.  
\end{proposition}
\begin{proof} We may assume that $A_i$ and $B_i+C_i$ have no common components.  Replacing
$B_i$ by $B_i-\rfdown B_i.$ and $A_i$ by $A_i+\rfdown B_i.$ we may assume that $\rfdown
\Delta_i.=\rfdown A_i.$.  As the coefficients of $A_i+B_i$ belong to a set which satisfies
the DCC, \eqref{t_num} implies that not all of the coefficients of $C_i$ are increasing.
In particular at least one coefficient of $C_i$ is bounded away from one.

Let $a_i$ be the log discrepancy of $(X_i,\Delta_i)$.

\textbf{Case A:} $\lim a_i>0$.  

In this case, we assume that $a_i$ is bounded away from zero.  

\textbf{Case A, Step 1:} We reduce to the case $X_i$ is $\mathbb{Q}$-factorial and the
Picard number of $X_i$ is one.

As we are assuming that $a_i$ is bounded away from zero, $A_i=0$ and so $(X_i,\Delta_i)\in
\mathfrak{N}_n(I,c_i)$, so that $c_i\in N_n(I)=K_n(I)$, by \eqref{l_complex}.  Thus we may
assume that $(X_i,\Delta_i)\in \mathfrak{K}_m(I,c_i)$, for some $m\leq n$.  If $m<n$ then
we are done by induction.  Otherwise we may assume that $X_i$ is $\mathbb{Q}$-factorial
and the Picard number of $X_i$ is one.

Possibly passing to a subsequence, \eqref{l_number} implies that we may assume that the
number of components of $B_i$ and $C_i$ is fixed.  As the only accumulation point of
$D(I)$ is one and the coefficients of $B_i$ are bounded away from one, possibly passing to
a subsequence we may assume that the coefficients of $B_i$ are fixed and that the
coefficients of $C_i$ have the form
$$
\frac{r-1}r+\frac fr+\frac {kc_i}r,
$$
where $k$, $r$ and $f$ depend on the component but not on $i$. 

Given $t\in [0,1]$, let $C_i(t)$ be the divisor with the same components as $C_i$ but now
with coefficients
$$
\frac{r-1}r+\frac fr+\frac {kt}r,
$$
so that $C_i=C_i(c_i)$.  Let
$$
h_i=\sup \{\, t \,|\, \text{$(X_i,B_i+C_i(t))$ is log canonical} \,\},
$$
be the log canonical threshold.  Set $h=\lim h_i$.  

\textbf{Case A, Step 2:} We reduce to the case $h>c$. 

Suppose that $h\leq c$.  As $c_i\leq h_i$, it follows that $h=c$.  
Now
$$
h_i\in \Lct_n(D(I))=N_{n-1}(I),
$$
so that we are done by induction in this case.  

\textbf{Case A, Step 3:} We reduce to the case $\vol(X_i,C_i)$ is unbounded.  

Suppose not, suppose that $\vol(X_i,C_i)$ is bounded from above.  Let
$$
d_i=\frac{c_i+h_i}2 \qquad \text{and} \qquad d=\frac{c+h}2.
$$
Then the coefficients of $(X_i,B_i+C_i(d))$ are fixed.  The log discrepancy of
$(X_i,B_i+C_i(d_i))$ is at least $a_i/2$ so that the log discrepancy of $(X_i,B_i+C_i(d))$
is bounded away from zero.  As $h>c$, possibly passing to a tail of the sequence, we may
assume that $d>c_i$ so that $K_{X_i}+B_i+C_i(d)$ is ample.  Note that
$$
\vol(X_i,K_{X_i}+B_i+C_i(d))=\vol(X_i,C_i(d)-C_i)
$$
is bounded from above by assumption.  \eqref{t_volume} implies that there is a positive
integer $m$ such that $\phi_{m(K_{X_i}+B_i+C_i(d))}$ is birational.  But then $\{\,
(X_i,\Delta_i) \,|\, i\in \mathbb{N} \,\}$ is log birationally bounded by
\cite[2.4.2.4]{HMX10}.  \eqref{t_ample-bounded} implies that $(X_i,\Delta_i)$ belongs to a
bounded family.  Thus we may find an ample Cartier divisor $H_i$ such that the
intersection numbers $T_i\cdot H_i^{n-1}$ and $-K_{X_i}\cdot H_i^{n-1}$ are bounded, where
$T_i$ is any component of $\Delta_i$.  Possibly passing to a subsequence, we may assume
that these intersection numbers are constant.  But then
$$
(K_{X_i}+\Delta_i)\cdot H_i^{n-1}=0, \qquad A_i\cdot H_i^{n-1}=0 \qquad \text{and} \qquad B_i\cdot H_i^{n-1} 
$$
is independent of $i$, whilst $C_i\cdot H_i^{n-1}$ is not constant, a contradiction.

\textbf{Case A, Step 4:} We finish case A.  

As $\vol(X_i,C_i)$ is unbounded and $C_i$ is ample, passing to a subsequence, we may find 
$g_i<c_i$ and a divisor
$$
0\leq \Theta_i\sim_{\mathbb{R}} C_i-C_i(g_i) \qquad \text{with} \qquad \lim g_i=c,
$$ 
such that $(X_i,\Phi_i=B_i+C_i(g_i)+\Theta_i)$ is purely log terminal but not kawamata log
terminal.  If $\phi\colon\map Y_i.X_i.$ is a divisorially log terminal modification then
$\phi$ extracts a unique prime divisor $S_i$ of log discrepancy zero with respect to
$(X_i,\Phi_i)$.  We may write
$$
K_{Y_i}+\Psi_i=\phi^*(K_{X_i}+\Phi_i)\qquad\text{and}\qquad K_{Y_i}+B'_i+C'_i+s_iS_i=\phi^*(K_{X_i}+\Delta_i),
$$ 
where $S_i=\rfdown\Psi_i.$, $B'_i$ and $C'_i$ are the strict transform of $B_i$ and $C_i$,
and $s_i<1$, as $(X_i,\Delta_i)$ is kawamata log terminal.

As $K_{Y_i}+\Psi_i$ is numerically trivial, $K_{Y_i}+\Psi_i-S_i$ is not pseudo-effective.
By \cite{BCHM06}, we may run $f\colon\rmap Y_i.W_i.$ the $(K_{Y_i}+\Psi_i-S_i)$-MMP until
we end with a Mori fibre space $\pi_i\colon\map W_i.Z_i.$.  As $K_{Y_i}+\Psi_i$ is
numerically trivial, every step of this MMP is $S_i$-positive, so that the strict
transform $T_i$ of $S_i$ dominates $Z_i$.  Let $F_i$ be the general fibre of $\pi_i$.
Replacing $Y_i$, $B'_i$, $C'_i$ and $\Psi_i$ by $F_i$ and the restriction of $f_*B'_i$,
$f_*C'_i$ and $f_*\Psi_i$ to $F_i$, we may assume that $S_i$, $\Psi_i$, $B'_i$ and $C'_i$
are multiples of the same ample divisor.  In particular $K_{Y_i}+B'_i+C'_i+S_i$ is ample.  

We let $C'_i(t)$ denote the strict transform of $C_i(t)$.  We may write 
$$
(K_{Y_i}+S_i+B'_i+C'_i(t))|_{S_i}=K_{S_i}+B''_i+C''_i(t),
$$
where the coefficients of $B''_i$ belong to $D(I)$ and the coefficients 
of $C''_i(t)\neq 0$ belong to $D_t(I)$.  We let $C''_i=C''_i(c_i)$.  

There are two cases.  Suppose that $(S_i,B''_i+C''_i)$ is not log canonical.  Let
$$
k_i=\sup \{\, t \,|\, \text{$(S_i,B''_i+C''_i(t))$ is log canonical} \,\},
$$
be the log canonical threshold.  Then $k_i\in \Lct_{n-1}(D(I))=N_{n-2}(I)$.  Then $k=\lim
k_i\in N_{n-2}(I)\subset N_{n-1}(I)$ by induction on $n$.  As $(S_i, B''_i+C''_i(g_i))$ is
kawamata log terminal, $k_i\in (g_i,c_i)$.  Thus
$$
c=\lim c_i=\lim k_i=k\in N_{n-1}(I).
$$

Otherwise we may suppose that $(S_i,B''_i+C''_i)$ is log canonical.  Let
$$
l_i=\sup \{\, t \,|\, \text{$(S_i,B''_i+C''_i(t))$ is pseudo-effective} \,\},
$$
be the pseudo-effective threshold.  Then $l_i\in N_{n-1}(I)$ and $l=\lim l_i\in
N_{n-1}(I)$ by induction on $n$.  On the other hand $l_i\in (g_i,c_i)$.  Thus
$$
c=\lim c_i=\lim l_i=l\in N_{n-1}(I).
$$

\textbf{Case B:} $\lim a_i=0$.

In this case, we assume that $a_i$ approaches $0$.  

\textbf{Case B, Step 1:} We reduce to the case $A_i\neq 0$, $X_i$ is
$\mathbb{Q}$-factorial and $(X_i,\Delta_i)$ is kawamata log terminal if and only if
$\rfdown A_i.=0$.

Possibly passing to a subsequence we may assume that $a_i\geq a_{i+1}$ and $a_i\leq 1$.
If $(X_i,\Delta_i)$ is not divisorially log terminal or $A_i\neq 0$ but $X_i$ is not
$\mathbb{Q}$-factorial then let $\pi_i\colon\map X'_i.X_i.$ be a divisorially log terminal
modification.  If $A_i=0$ then let $\pi_i\colon\map X'_i.X_i.$ extract a divisor of
minimal log discrepancy $a_i$, where $X'_i$ is $\mathbb{Q}$-factorial.  Either way, we may
write
$$
K_{X'_i}+\Delta'_i=\pi_i^*(K_{X_i}+\Delta_i),
$$
where $\Delta'_i$ is the strict transform of $\Delta_i$.  Let $B'_i$ and $C'_i$ be the
strict transforms of $B_i$ and $C_i$ and let $A'_i=\Delta'_i-B'_i-C'_i\neq 0$.  Then
$X'_i$ is a $\mathbb{Q}$-factorial projective variety of dimension $n$, $(X'_i,\Delta'_i)$
is a divisorially log terminal pair, $K_{X'_i}+\Delta'_i$ is numerically trivial, the
coefficients of $A'_i\neq 0$ are approaching one, the coefficients of $B'_i$ belong to
$D(I)$ and the coefficients of $C'_i\neq 0$ belong to $D_{c_i}(I)$.  Replacing
$(X_i,\Delta_i)$ by $(X'_i,\Delta'_i)$, we may assume that $A_i\neq 0$ and $X_i$ is
$\mathbb{Q}$-factorial.  Moreover $(X_i,\Delta_i)$ is kawamata log terminal if and only if
$\rfdown A_i.=0$.

\textbf{Case B, Step 2:} We are done if the support of $C_i$ and $\rfdown A_i.$ intersect.

Suppose that a component of $C_i$ intersects the normalisation of a component $S_i$ of
$\rfdown A_i.$.  Then we may write
$$
(K_{X_i}+\Delta_i)|_{S_i}=K_{S_i}+\Theta_i,
$$
by adjunction.  $S_i$ is projective of dimension $n-1$, $(S_i,\Theta_i)$ is log canonical,
$K_{S_i}+\Theta_i$ is numerically trivial, and we may write $\Theta_i=A'_i+B'_i+C'_i$,
where the coefficients of $A'_i$ approach one, the coefficients of $B'_i$ belong to $D(I)$
and the coefficients of $C'_i\neq 0$ belong to $D_{c_i}(I)$.  In this case, the limit $c$
belongs to $N_{n-2}(I)\subset N_{n-1}(I)$ by induction.

\textbf{Case B, Step 3:} We are done if $f_i\colon\map X_i.Z_i.$ is a Mori fibre 
space, $A_i$ dominates $Z_i$ and $\dim Z_i>0$.  

Let $F_i$ be the general fibre of $f_i$.  We may
write
$$
(K_{X_i}+\Delta_i)|_{F_i}=K_{F_i}+\Theta_i,
$$
by adjunction.  $F_i$ is projective of dimension at most $n-1$, $(F_i,\Theta_i)$ is log
canonical, $K_{F_i}+\Theta_i$ is numerically trivial, and we may write
$\Theta_i=A'_i+B'_i+C'_i$, where the coefficients of $A'_i$ approach one, the coefficients
of $B'_i$ belong to $D(I)$ and the coefficients of $C'_i$ belong to $D_{c_i}(I)$.

There are two cases.  Suppose that $C'_i=0$.  Then \eqref{t_num} implies that the
coefficients of $A'_i$ are fixed, so that $\rfdown A'_i.=A'_i$.  But then $\rfdown
A_i.\neq 0$ dominates $Z_i$.  On the other hand, as $C'_i=0$, $C_i$ does not intersect
$F_i$, that is, $C_i$ does not dominate $Z_i$.  But then $C_i$ must contain a fibre so
that $A_i$ and $C_i$ intersect and we are done by Case B, Step 2.  Otherwise $C'_i\neq 0$.
In this case $c_i\in N_{n-1}(I)$ so that
$$
c=\lim c_i\in N_{n-2}(I)\subset N_{n-1}(I),
$$
by induction.

\textbf{Case B, Step 4:} We reduce to the case $(X_i,\Delta_i)$ is kawamata log terminal.  

Suppose not, suppose that $(X_i,\Delta_i)$ is not kawamata log terminal.  By Case B, Step
1, this implies that $S_i=\rfdown A_i.$ is not the zero divisor.  Let
$\Theta_i=\Delta_i-S_i$.  We run the $(K_{X_i}+\Theta_i)$-MMP with scaling of some ample
divisor.  Let $f_i\colon\rmap X_i.X'_i.$ be a step of the $(K_{X_i}+\Theta_i)$-MMP.  As
$K_{X_i}+\Delta_i$ is numerically trivial, $f_i$ is automatically $S_i$-positive.  Let
$A'_i=f_{i*}A_i$, $B'_i=f_{i*}B_i$ and $C'_i=f_{i*}C_i$.  First suppose that $f_i$ is
birational.  If $C'_i=0$ then \eqref{t_num} implies that the coefficients of $A'_i$ are
all one.  As $f_i$ contracts $C_i$ it does not contract a component of $A_i$ and so it
follows that the coefficients of $A_i$ are all one, that is, $S_i=A_i$.  As $f_i$
contracts $C_i$ and $f_i$ is $S_i$-positive, $C_i$ intersects $S_i$ and we are done by
Case B, Step 2.  Therefore we may assume that $C'_i\neq 0$ and we may replace
$(X_i,\Delta_i)$ by $(X'_i,\Delta'_i)$.  As the MMP must terminate with a Mori fibre
space, replacing $(X_i,\Delta_i)$ with $(X'_i,\Delta'_i)$ finitely many times, we may
assume that $f_i\colon\map X_i.Z_i=X'_i.$ is a Mori fibre space and $S_i$ dominates $Z_i$.
By Case B, Step 3, we may assume that $Z_i$ is a point.  But then the support of $S_i$ and
$C_i$ intersect and we are done by Case B, Step 2.

\textbf{Case B, Step 5:} We reduce to the case $X_i$ has Picard number one.    

We run the $(K_{X_i}+B_i+C_i)$-MMP with scaling of some ample divisor.  Let
$f_i\colon\rmap X_i.X'_i.$ be a step of the $(K_{X_i}+B_i+C_i)$-MMP.  As
$K_{X_i}+\Delta_i$ is numerically trivial $f_i$ is automatically $A_i$-positive.  Let
$A'_i=f_{i*}A_i$, $B'_i=f_{i*}B_i$ and $C'_i=f_{i*}C_i$.  First suppose that $f_i$ is
birational.  Suppose $C'_i=0$.  As $f_i$ contracts only one divisor and $A_i$ and $C_i$
are non-zero by assumption, it follows that $A'_i\neq 0$.  \eqref{t_num} implies that the
coefficients of $A'_i$ are all one, which contradicts the fact that $(X_i,\Delta_i)$ is
kawamata log terminal.  Therefore we may assume that $C'_i\neq 0$ and we may replace
$(X_i,\Delta_i)$ by $(X'_i,\Delta'_i)$.  As the MMP must terminate with a Mori fibre
space, replacing $(X_i,\Delta_i)$ with $(X'_i,\Delta'_i)$ finitely many times, we may
assume that $f_i\colon\map X_i.Z_i=X'_i.$ is a Mori fibre space and $A_i$ dominates $Z_i$.

By Case B, Step 3 we may assume that $Z_i$ is a point, so that $X_i$ has Picard number
one.

\textbf{Case B, Step 6:} We finish case B and the proof.   

Possibly passing to a subsequence, \eqref{l_number} implies that we may assume that the
number of components of $B_i$ and $C_i$ is fixed.  As the only accumulation point of
$D(I)$ is one and the coefficients of $B_i$ are bounded away from one, possibly passing to
a subsequence we may assume that the coefficients of $B_i$ are fixed and that the
coefficients of $C_i$ have the form
$$
\frac{r-1}r+\frac fr+\frac {kc_i}r,
$$
where $k$, $r$ and $f$ depend on the component but not on $i$. 

Given $t\in [0,1]$, let $C_i(t)$ be the divisor with the same components as $C_i$ but now
with coefficients
$$
\frac{r-1}r+\frac fr+\frac {kt}r,
$$
so that $C_i=C_i(c_i)$. 

Let $T_i$ be the sum of the components of $A_i$, so that $T_i$ has the same components as
$A_i$ but now every component has coefficient one.  Then $A_i\leq T_i$ and $C_i(c)\leq
C_i$.  Note that $(X_i,A_i+B_i+C_i(c))$ is kawamata log terminal as $(X_i,A_i+B_i+C_i)$ is
kawamata log terminal.  As the coefficients of $A_i+B_i+C_i(c)$ belong to a set which
satisfies the DCC, possibly passing to a tail of the sequence, \eqref{t_simple} implies
that $(X_i,T_i+B_i+C_i(c))$ is log canonical.

Suppose that $(X_i,T_i+B_i+C_i)$ is not log canonical.  Let 
$$
d_i=\sup \{\, t\in [c,c_i) \,|\, \text{$(X_i,T_i+B_i+C_i(t))$ is log canonical} \,\},
$$
be the log canonical threshold.  Then $d_i\in \Lct_n(D(I))=N_{n-1}(I)$ and $c=\lim d_i$
and so we are done by induction on the dimension.

Thus we may assume that $(X_i,T_i+B_i+C_i)$ is log canonical.  Let 
$$
e_i=\sup \{\, t\in \mathbb{R} \,|\, \text{$K_{X_i}+T_i+B_i+C_i(t))$ is pseudo-effective} \,\},
$$
be the pseudo-effective threshold.  Suppose that $e_i<c$.  Let 
$$
f_i=\sup \{\, t\in \mathbb{R} \,|\, \text{$K_{X_i}+tT_i+B_i+C_i(c))$ is pseudo-effective} \,\},
$$
be the pseudo-effective threshold.  As $e_i<c$, $f_i<1$ and $\lim f_i=1$, so that the
coefficients of $f_iT_i+B_i+C_i(c)$ belong to a set which satisfies the DCC, which
contradicts \eqref{t_num}.  Thus $e_i\geq c$.  On the other hand $e_i<c_i$ as
$K_{X_i}+T_i+B_i+C_i$ is strictly bigger than $K_{X_i}+A_i+B_i+C_i$, which is numerically
trivial.  Thus $\lim e_i=c$.  Possibly passing to a subsequence we may assume that either
$e_i>e_{i+1}$ for all $i$ or $e_i=c$.  In the former case we might as well replace
$C_i=C_i(c_i)$ by $C_i(e_i)$.  In this case some component of $C_i$ intersects a component
$S_i$ of $T_i$ and we are done by Case B, Step 2.  In the latter case we restrict to a
component $S_i$ of $T_i$ and apply adjunction to conclude that $c=e_i\in N_{n-1}(I)$.
\end{proof}

\begin{proof}[Proof of \eqref{t_accumulation}] By \eqref{p_local-global} it suffices to
prove that the accumulation points of $N_{n}(I)$ belong to $N_{n-1}(I)$.  Suppose that
$\ilist c.\in [0,1]$ is a strictly decreasing sequence of real numbers such that
$\mathfrak{N}(I,c_i)$ is non-empty.  Pick $(X_i,\Delta_i)\in \mathfrak{N}(I,c_i)$.  By
assumption we may write $\Delta_i=B_i+C_i$ where the coefficients of $B_i$ belong to
$D(I)$ and the coefficients of $C_i\neq 0$ belong to $D_{c_i}(I)$, and so
\eqref{p_stronger} implies that the limit $c$ belongs to $N_{n-1}(I)$.
\end{proof}

\bibliographystyle{hamsplain}
\bibliography{/home/mckernan/Jewel/Tex/math}

%\addcontentsline{toc}{section}{References}

\end{document}